\RequirePackage{fix-cm}

\documentclass[nospthms]{svjour3}       
\newtheorem{theorem}{\indent\sc Theorem}[section]

\newtheorem{corollary}[theorem]{\indent\sc Corollary}
\newtheorem{proposition}[theorem]{\indent\sc Proposition}

\newtheorem{definition}[theorem]{\indent\sc Definition}
\newtheorem{remark}[theorem]{\indent\sc Remark}
\newtheorem{example}[theorem]{\indent\sc Example}

\newenvironment{proof}{\paragraph{Proof:}}{\hfill$\square$}

\usepackage{graphicx}
\usepackage{amssymb}
\usepackage{latexsym,amsmath,graphicx}
\usepackage[all,web]{xy}

\usepackage{hyperref}

\def\urlfont{\DeclareFontFamily{OT1}{cmtt}{\hyphenchar\font='057}
              \normalfont\ttfamily \hyphenpenalty=10000}

\DeclareFontFamily{OT1}{rsfs10}{}
\DeclareFontShape{OT1}{rsfs10}{m}{n}{ <-> rsfs10 }{}
\DeclareMathAlphabet{\mathscript}{OT1}{rsfs10}{m}{n}

\DeclareMathOperator{\Hom}{Hom}     
\DeclareMathOperator{\Tors}{Tors}    
\DeclareMathOperator{\Pic}{Pic}     
\DeclareMathOperator{\Cl}{Cl}       
\DeclareMathOperator{\rk}{rk}       
\DeclareMathOperator{\coker}{coker} 
\DeclareMathOperator{\Sing}{Sing}   
\DeclareMathOperator{\codim}{codim} 
\DeclareMathOperator{\diag}{diag}   
\DeclareMathOperator{\fan}{fan}     
\DeclareMathOperator{\lcm}{lcm}     
\DeclareMathOperator{\HNF}{HNF}     
\DeclareMathOperator{\SNF}{SNF}     
\DeclareMathOperator{\REF}{REF}     

\newtheorem{thm-def}[theorem]{Theorem--Definition}

\newcommand{\oneline}{\vskip12pt}
\newcommand{\halfline}{\vskip6pt}
\def \a{\alpha }
\def \b{\beta }
\def \d{\delta }

\def \Ga{\Gamma }
\def \Si{\Sigma }
\def \g{\gamma}

\def \v{\mathbf{v}}

\def \t{\mathbf{t}}
\def \x{\mathbf{x}}

\def \1{\mathbf{1}}
\def \0{\mathbf{0}}

\def\P{{\mathbb{P}}}

\def\p2{\mathbb{P}^2}
\def\p3{\mathbb{P}^3}
\def\p4{\mathbb{P}^4}

\def\rk{\operatorname{rk}}

\def\GL{\operatorname{GL}}

\def\Z{\mathbb{Z}}

\def\C{\mathbb{C}}
\def\R{\mathbb{R}}

\def\Q{\mathbb{Q}}
\def\N{\mathbb{N}}

\def\SF{\mathcal{SF}}

\def\G{\mathcal{G}}
\def\Ga{\Gamma}

\def\Weil{\mathcal{W}_T}
\def\Cart{\mathcal{C}_T}

\begin{document}

\title{A $\Q$--factorial complete toric variety is a quotient of a poly weighted space\thanks{The authors were partially supported by the MIUR-PRIN 2010-11 Research Funds ``Geometria delle Variet\`{a} Algebriche''. The first author is also supported by the I.N.D.A.M. as a member of the G.N.S.A.G.A.}
}


\author{Michele Rossi         \and
        Lea Terracini 
}


\institute{Michele Rossi \at
              Dipartimento di Matematica - Universit\`a di Torino - Via Carlo Alberto 10 - 10123 Torino (Italy) \\
              Tel.: +39-011-6702813\\
              Fax: +39-011-6702878\\
              \email{michele.rossi@unito.it}           
           \and
          Lea Terracini \at
                        Dipartimento di Matematica - Universit\`a di Torino - Via Carlo Alberto 10 - 10123 Torino (Italy) \\
                        Tel.: +39-011-6702813\\
                        Fax: +39-011-6702878\\
                        \email{lea.terracini@unito.it}
}


\maketitle

\begin{abstract}
We prove that every $\Q$--factorial complete toric variety is a finite abelian quotient of a poly weighted space (PWS), as defined in our previous work \cite{RT-LA&GD}. This generalizes the Batyrev--Cox and Conrads de\-scrip\-tion of a $\Q$--factorial complete toric variety of Picard number 1, as a finite quotient of a weighted projective space (WPS) \cite[Lemma~2.11]{BC} and \cite[Prop.~4.7]{Conrads}, to every possible Picard number, by replacing the covering WPS with a PWS. By Buczy\'nska's results \cite{Buczynska}, we get a universal picture of coverings in codimension 1 for every $\Q$--factorial complete toric variety, as  topological counterpart of the $\Z$--linear universal property of the double Gale dual of a fan matrix.

As a consequence we describe the bases of the subgroup of Cartier divisors inside the free group of Weil divisors and the bases of the Picard subgroup inside the class group, respectively, generalizing to every $\Q$--factorial complete toric variety the description given in \cite[Thm.~2.9]{RT-LA&GD} for a PWS.
\end{abstract}

\tableofcontents

\section*{Introduction}

The present paper is the second part of a longstanding tripartite study aimed at realizing, for $\Q$--factorial projective toric varieties, a classification inspired by what V.~Batyrev did in 1991 for smooth complete toric varieties \cite{Batyrev91}. The first part of this study is given by \cite{RT-LA&GD}, in which we studied Gale duality from the $\Z$--linear point of view and defined poly weighted spaces (PWS, for short: see the following Definition \ref{def:PWS}). The reader will often be referred to this work for notation, preliminaries and results.

In this paper, as a first application of $\Z$--linear Gale duality, we exhibit
\begin{itemize}
  \item[(i)] every $\Q$--factorial complete toric variety as a finite abelian quotient of a PWS (Theorem \ref{thm:covering&quotient}),
  \item[(ii)] a universal covering in codimension 1 theorem for every $\Q$--factorial complete toric variety (Corollary \ref{cor:universale}),
  \item[(iii)] explicit bases of the subgroup of Cartier divisors inside the free group of Weil divisors and of the Picard subgroup inside the class group, for every $\Q$--factorial complete toric variety (Theorem \ref{thm:generazione}).
\end{itemize}
The first result (i) generalizes, to every $\Q$--factorial complete toric variety, results \cite[Lemma~2.11]{BC} and \cite[Prop.~4.7]{Conrads} by V.~Batyrev--D.~Cox and H.~Conrads, respectively, in terms of PWS. Roughly speaking, by adopting the usual terminology for toric varieties of Picard number 1, called \emph{fake WPS} as suitable finite quotients of a weighted projective space (WPS, for short), such a generalization can be stated by saying that \emph{every $\Q$--factorial complete toric variety is a fake PWS}. Section \ref{sez:1covering} is entirely devoted to explaining this fact. The main result is given by Theorem~\ref{thm:covering&quotient}: the group action realizing the quotient is determined by the torsion subgroup $\Tors(\Cl(X))$ of the classes group $\Cl(X)$ of the given $\Q$--factorial complete toric variety $X$, and it is represented by a suitable \emph{torsion matrix} $\Ga$. In other words the codimension 1 structure of the $\Q$--factorial complete toric variety $X$ can be completely assigned by a \emph{weight matrix} $Q$, in the sense of Definition \ref{def:Wmatrice}, describing the covering PWS as a Cox's quotient, and by the torsion matrix $\Ga$.

The second result (ii) is obtained by the former (i), by making  significant use of W.~Buczynska's results presented in  \cite{Buczynska} and  briefly summarized in \ref{ssez:Buczynska}, since a PWS can be characterized as a \emph{1--connected in codimension 1} $\Q$--factorial complete toric variety. By $\Z$--linear Gale duality and the universal property (\ref{universalità}), we get a complete description of any covering in codimension 1 of a given $\Q$--factorial complete toric variety: this is Corollary \ref{cor:universale}.

 The so given description, of a $\Q$--factorial complete toric variety as a fake PWS, allows us to extend the description of the bases of the subgroup of Cartier divisors inside the free group of Weil divisors and of the Picard subgroup inside the class group, given in \cite[Thm.~2.9]{RT-LA&GD} for a PWS, to every $\Q$--factorial complete toric variety. This is the third result (iii): Section \ref{sez:W&C} is entirely devoted to this purpose. The main result in this context is given by Theorem~\ref{thm:generazione}, whose proof is essentially the same as that of \cite[Thm.~2.9]{RT-LA&GD}, although the torsion in $\Cl(X)$ makes the situation significantly more intricate.

For the most part, results stated in Theorem \ref{thm:generazione} are based on the algorithms giving Hermite and Smith normal forms of a matrix ($\HNF$ and $\SNF$, respectively), and associated switching matrices, which are well known algorithms (see e.g. \cite[Algorithms 2.4.4 and 2.4.14]{Cohen}) implemented in many computer algebra procedures. Theorem \ref{thm:generazione} gives quite effective and constructive me\-thods for producing large amounts of interesting information characterizing a given $\Q$--factorial complete toric variety $X$ (see Remark \ref{rem:} for further details).

Finally section \ref{sez:fan} is dedicated to reconstructing a fan matrix $V$, in the sense of Definition \ref{def:Fmatrice}, of a $\Q$--factorial complete toric variety $X$, assigned by means of a weight matrix $Q$ and a torsion matrix $\Ga$. Here, the main result is given by Theo\-rem~\ref{thm:da-quot-a-fan}, in which the fan matrix is obtained by Gale duality and by reconstructing the matrix $\beta$ characterizing the universal property (\ref{universalità}) and linking a fan matrix of $X$ with a fan matrix of the covering PWS.

The last section \ref{sez:Esempi} is devoted to giving several examples and applications of all the techniques described in the previous sections: examples are presented following the lines stated in Remark 3.3. Here it is rather important for the reader to be equipped with some computer algebra package which has the ability to produce Hermite and Smith normal forms of matrices and their switching matrices. For example, using Maple, similar procedures are given by \texttt{HermiteForm} and \texttt{SmithForm}  with their output options.

\section{Preliminaries and notation}

The present paper gives a first application of the $\Z$--linear Gale Duality developed in the previous paper \cite{RT-LA&GD}, to which the reader is referred for notation and all the necessary preliminary results. In particular concerning toric varieties, cones and fans, the reader is referred to \cite[\S~1.1]{RT-LA&GD}, for linear algebraic preliminaries about normal forms of matrices (Hermite and Smith normal forms - HNF and SNF for short) to \cite[\S~1.2]{RT-LA&GD}. $\Z$--linear Gale Duality and what is concerning fan matrices ($F$--matrices) and weight matrices ($W$-matrices) is developed in \cite[\S~3]{RT-LA&GD}.

Every time the needed nomenclature will be recalled either directly by giving the necessary definition or by reporting the precise reference. Here is a list of main notation and relative references:

\subsection{List of notation}\label{ssez:lista}
Let $X(\Si)$ be a $n$--dimensional toric variety and $T\cong(\C^*)^n$ the acting torus, then
\begin{eqnarray*}
  &M,N,M_{\R},N_{\R}& \text{denote the \emph{group of characters} of $T$, its dual group and}\\
  && \text{their tensor products with $\R$, respectively;} \\
  &\Si\subseteq \mathfrak{P}(N_{\R})& \text{is the fan defining $X$; $\mathfrak{P}(N_{\R})$ denotes the power set of $N_{\R}$} \\
  &\Si(i)& \text{is the \emph{$i$--skeleton} of $\Si$, that is, the collection of all the}\\
  && \text{$i$--dimensional cones in $\Si$;} \\
  &\langle\v_1,\ldots,\v_s\rangle\subseteq\N_{\R}& \text{denotes the cone generated by the vectors $\v_1,\ldots,\v_s\in N_{\R}$;}\\
  && \text{if $s=1$ then this cone is also called the \emph{ray} generated by $\v_1$;} \\
  &\mathcal{L}(\v_1,\ldots,\v_s)\subseteq N& \text{denotes the sublattice spanned by $\v_1,\ldots,\v_s\in N$\,;}\\
&\mathcal{W}_T(X),\  \mathcal{C}_T(X),\  \mathcal{P}_T(X) &  \text{denote torus invariant  Weil, Cartier, principal divisors of $X$, resp.}
\end{eqnarray*}

Let $A\in\mathbf{M}(d,m;\Z)$ be a $d\times m$ integer matrix, then
\begin{eqnarray*}
  &\mathcal{L}_r(A)\subseteq\Z^m& \text{denotes the sublattice spanned by the rows of $A$;} \\
  &\mathcal{L}_c(A)\subseteq\Z^d& \text{denotes the sublattice spanned by the columns of $A$;} \\
  &A_I\,,\,A^I& \text{$\forall\,I\subseteq\{1,\ldots,m\}$ the former is the submatrix of $A$ given by}\\
  && \text{the columns indexed by $I$ and the latter is the submatrix of}\\
  && \text{$A$ whose columns are indexed by the complementary }\\
  && \text{subset $\{1,\ldots,m\}\backslash I$;} \\
  &_sA\,,\,^sA& \text{$\forall\,1\leq s\leq d$\ the former is the submatrix of $A$ given by the}\\
  && \text{lower $s$ rows and the latter is the submatrix of $A$ given by}\\
  && \text{the upper $s$ rows of $A$;} \\
  &\HNF(A)\,,\,\SNF(A)& \text{denote the Hermite and the Smith normal forms of $A$,}\\
  && \text{respectively.}\\
  &\REF& \text{Row Echelon Form of a matrix;}\\
  &\text{\emph{positive}}& \text{a matrix (vector) whose entries are non-negative.}
\end{eqnarray*}
Given a $F$--matrix $V=(\v_1,\ldots,\v_{n+r})\in\mathbf{M}(n,n+r;\Z)$ (see Definition \ref{def:Fmatrice} below), then
\begin{eqnarray*}
  &\langle V\rangle=\langle\v_1,\ldots,\v_{n+r}\rangle\subseteq N_{\R}& \text{denotes the cone generated by the columns of $V$;} \\
  &\SF(V)=\SF(\v_1,\ldots,\v_{n+r})& \text{is the set of all rational simplicial fans $\Si$ such that}\\
  && \text{$\Sigma(1)=\{\langle\v_1\rangle,\ldots,\langle\v_{n+r}\rangle\}\subseteq N_{\R}$ and $|\Si|=\langle V\rangle$ \cite[Def.~1.3]{RT-LA&GD}.}\\
  &\G(V)&=Q\ \text{is a \emph{Gale dual} matrix of $V$ \cite[\S~3.1]{RT-LA&GD};} \\
\end{eqnarray*}
\oneline
\noindent Let us start by recalling four fundamental definitions:

\begin{definition} A $n$--dimensional $\Q$--factorial complete toric variety $X=X(\Si)$ of rank $r$ is the toric variety defined by a $n$--dimensional \emph{simplicial} and \emph{complete} fan $\Si$ such that $|\Si(1)|=n+r$ \cite[\S~1.1.2]{RT-LA&GD}. In particular the rank $r$ coincides with the Picard number i.e. $r=\rk(\Pic(X))$.
\end{definition}

\begin{definition}[\cite{RT-LA&GD}, Def.~3.10]\label{def:Fmatrice} An \emph{$F$--matrix} is a $n\times (n+r)$ matrix  $V$ with integer entries, satisfying the conditions:
\begin{itemize}
\item[(a)] $\rk(V)=n$;
\item[(b)] $V$ is \emph{$F$--complete} i.e. $\langle V\rangle=N_{\R}\cong\R^n$ \cite[Def.~3.4]{RT-LA&GD};
\item[(c)] all the columns of $V$ are non zero;
\item[(d)] if ${\bf  v}$ is a column of $V$, then $V$ does not contain another column of the form $\lambda  {\bf  v}$ where $\lambda>0$ is real number.
\end{itemize}
A \emph{$CF$--matrix} is a $F$-matrix satisfying the further requirement
\begin{itemize}
\item[(e)] the sublattice ${\mathcal L}_c(V)\subseteq\Z^n$ is cotorsion free, that is, ${\mathcal L}_c(V)=\Z^n$ or, equivalently, ${\mathcal L}_r(V)\subseteq\Z^{n+r}$ is cotorsion free.
\end{itemize}
A $F$--matrix $V$ is called \emph{reduced} if every column of $V$ is composed by coprime entries \cite[Def.~3.13]{RT-LA&GD}.
\end{definition}
\begin{example}
{\rm The most significant example of $F$-matrix is given by a matrix $V$ whose columns are    integral vectors generating the rays of the $1$-skeleton $\Sigma(1)$ of a rational fan $\Sigma$. In the following $V$ will be called a \emph{fan matrix} of $\Sigma$; when every column of $V$ is composed by coprime entries, it will be called a \emph{reduced fan matrix}. For a detailed definition see
\cite[Def.~1.3]{RT-LA&GD} }.
\end{example}
\begin{definition}[\cite{RT-LA&GD}, Def.~3.9]\label{def:Wmatrice} A \emph{$W$--matrix} is an $r\times (n+r)$ matrix $Q$  with integer entries, satisfying the following conditions:
\begin{itemize}
\item[(a)] $\rk(Q)=r$;
\item[(b)] ${\mathcal L}_r(Q)$ does not have cotorsion in $\Z^{n+r}$;
\item[(c)] $Q$ is \emph{$W$--positive}, that is, $\mathcal{L}_r(Q)$ admits a basis consisting of positive vectors \cite[Def.~3.4]{RT-LA&GD}.
\item[(d)] Every column of $Q$ is non-zero.
\item[(e)] ${\mathcal L}_r(Q)$   does not contain vectors of the form $(0,\ldots,0,1,0,\ldots,0)$.
\item[(f)]  ${\mathcal L}_r(Q)$ does not contain vectors of the form $(0,a,0,\ldots,0,b,0,\ldots,0)$, with $ab<0$.
\end{itemize}
A $W$--matrix is called \emph{reduced} if $V=\G(Q)$ is a reduced $F$--matrix \cite[Def.~3.14, Thm.~3.15]{RT-LA&GD}
\end{definition}

\begin{definition}[\cite{RT-LA&GD} Def.~2.7]\label{def:PWS} A \emph{poly weighted space} (PWS) is a $n$--dimensional $\Q$--factorial complete toric variety $X(\Si)$ of rank $r$, whose reduced fan matrix $V$ is a $CF$--matrix i.e. if
\begin{itemize}
  \item $V=(\v_1,\ldots,\v_{n+r})$ is a $n\times(n+r)$ $CF$--matrix,
  \item $\Si\in\SF(V)$.
\end{itemize}
\end{definition}

\subsection{The fundamental group in codimension 1}\label{ssez:Buczynska}

This subsection is devoted to recall notation and results introduced in \cite[\S~3]{Buczynska} to which the interested reader is referred for any further detail.

\begin{definition}\label{def:pi11} The \emph{fundamental group in codimension $1$} of an irreducible, complex, algebraic variety $X$ is the inverse limit
\begin{equation*}
    \pi_1^1(X):= \varprojlim_{U\subseteq X} \pi_1(U)
\end{equation*}
where $U$ is an open, non-empty, algebraic subset of $X$ such that $\codim_{X}(X\setminus U)\geq 2$. $X$ is called \emph{$1$--connected in codimension $1$} if $\pi_1^1(X)$ is trivial.
\end{definition}

\begin{theorem}[\cite{Buczynska}, Thm. 3.4]\label{thm:pi11=pi1} Assume $X$ be a smooth variety and $V\subseteq X$ be a closed subset  with $\codim_{X}V\geq 2$. Then $\pi_1(X)=\pi_1(X\setminus V)$.
\end{theorem}

\begin{corollary}[\cite{Buczynska}, Cor.3.9 and Cor.3.10] If $X$ is a normal variety then
$$\pi_1^1(X)=\pi_1(X\setminus \Sing(X))\ .$$
In particular if $X$ is smooth then $\pi_1^1(X)=\pi_1(X)$.
\end{corollary}

\noindent In fact, by Theorem \ref{thm:pi11=pi1}, the inverse limit in the Definition \ref{def:pi11} has a realization on the smooth subset $X\setminus \Sing(X)$.

\begin{definition}\label{def:1-covering} A finite surjective morphism $\varphi:Y\rightarrow X$ is called a \emph{covering in codimension $1$} (or simply a \emph{$1$--covering}) if it is unramified in codimesion $1$, that is, there exists a subvariety $V\subseteq X$ such that $\codim_X V\geq 2$ and $\varphi|_{Y_V}$ is a topological covering, where $Y_V:=\varphi^{-1}(X\setminus V)$. Moreover a \emph{universal covering in codimension $1$} is a $1$-covering $\varphi:Y\rightarrow X$ such that for any $1$--covering $\phi:X'\rightarrow X$ of $X$ there exists a $1$--covering $f:Y\rightarrow X'$ such that $\varphi=\phi\circ f$.
\end{definition}

\begin{proposition}[\cite{Buczynska}, Rem. 3.14] A $1$--covering $\varphi:Y\rightarrow X$ is universal if and only if $Y$ is 1-connected in codimension 1 i.e. $\pi_1^1(Y)=0$.
\end{proposition}

\begin{theorem}[$\pi_1^1$ for toric varieties \cite{Buczynska}, Thm. 4.8]\label{thm:Buczynska} Let $X=X(\Si)$ with $\Si$ be a fan in $N_{\R}$. For any ray $\rho\in\Si(1)$ let $\v_{\rho}$ be the generator of the monoid $\rho\cap N$ and consider the $\Z$-module
\begin{equation}\label{NSigma}
    N_{\Si(1)}:=\mathcal{L}\left(\v_{\rho}\ |\ \rho\in\Si(1)\right)
\end{equation}
as a subgroup of the lattice $N$. Then $\pi^1_1(X)\cong N/N_{\Si(1)}$\ .
\end{theorem}
\noindent This result is an application of Van Kampen Theorem to the open covering $\{X_{\rho}\}_{\rho\in\Si(1)}$ of $X$, where $X_{\rho}=X(\rho)\cong\C\times(\C^*)^{n-1}$ is the toric variety associated with the fan given by the single ray $\rho\in\Si$.

\section{Universal 1--coverings}\label{sez:1covering}

Jointly with \cite[Thm. 2.4, Prop.~2.6]{RT-LA&GD}, the previous Theorem \ref{thm:Buczynska} gives the following e\-qui\-va\-lence between a 1--connected $\Q$--factorial complete variety and a PWS:

\begin{theorem}\label{thm:equivalenze} Let $X=X(\Si)$ be a $\Q$--factorial complete $n$--dimensional toric variety of rank $r$. Then the following are equivalent:
\begin{enumerate}
  \item $X$ is 1-connected in codimension 1,
  \item $\pi^1_1(X)\cong N/N_{\Si(1)}\cong \Tors(\Cl(X)) =0 $,
  \item the reduced fan matrix $V$ of $X$ has coprime $n\times n$ minors,
  \item the $\HNF$ of the transposed matrix $V^T$ is given by $\left(
                                                                 \begin{array}{c}
                                                                   I_n \\
                                                                   \mathbf{0}_{r,n} \\
                                                                 \end{array}
                                                               \right)
  $,
  \item $X$ is a PWS.
\end{enumerate}
\end{theorem}

The previous results allow us to sketch a nice geometric picture associated with any $\Q$--factorial, complete toric variety: the following Theorem \ref{thm:covering&quotient} is a generalization on the rank $r$ of a well known result holding for $r=\rk(\Pic (X))=1$ \cite[Lemma~2.11]{BC}, \cite[Prop.~4.7]{Conrads}. In fact a 1-connected in codimension 1, $\Q$--factorial, complete $n$--dimensional toric variety of rank 1 is necessarily a \emph{weighted projective space} (WPS) whose weights are given by the $1\times (n+1)$ weight matrix $Q$.

\begin{theorem}\label{thm:covering&quotient} A $\Q$--factorial, complete toric variety $X$ admits a canonical universal 1-covering, $Y$ which is a PWS and such that the 1-covering morphism $\varphi:Y\rightarrow X$ is equivariant with respect to the torus actions. In particular every $\Q$--factorial, complete toric variety $X$ can be canonically described as a finite geometric quotient $X\cong Y/\pi^1_1(X)$ of a PWS $Y$ by the torus--equivariant action of $\pi^1_1(X)$.
\end{theorem}

\begin{proof}
Given a complete fan $\Sigma\subseteq \mathfrak{P}(N_{\R})$, let $X(\Sigma)$ be the associated toric variety. Let $\widehat{N}$ be the sublattice $N_{\Sigma(1)}\subseteq N$ defined in (\ref{NSigma}); it has finite index since $\Sigma$ is complete.  Let $\widehat{\Sigma}$ be the fan defined by $\Sigma\subseteq \mathfrak{P}(\widehat{N}_{\R})$. Consider the toric variety $Y(\widehat{\Sigma})$.
The inclusion $i:\widehat{N}\hookrightarrow N$ induces a surjection $i_*: Y\rightarrow X$ which turns out to be the canonical projection on the quotient of $Y$ by the action of $N/\widehat{N}$. Furthermore, part 2 of Theorem \ref{thm:equivalenze} shows  that $\Cl(Y)$ is torsion-free, so that the results of
Section \ref{ssez:Buczynska} imply that the canonical morphism $i_*$ is the universal $1$-covering. Finally,
if $X$ is simplicial, so is $Y$, and then $Y$ is a PWS by Theorem \ref{thm:equivalenze}. This completes the proof.
\end{proof}

\begin{remark}{\rm Let us notice that the given proof of Theorem \ref{thm:covering&quotient} shows that a $\Q$--factorial complete toric variety $X$ uniquely determines its universal 1-covering $Y$ which is the PWS defined by the following data:
\begin{itemize}
  \item a weight matrix $Q=\G(V)$, where $V$ is a reduced fan matrix of $X$: in fact a fan matrix of $Y$ is given by $\G(Q)=\widehat{V}$,
  \item the choice of a fan $\widehat{\Si}\in \mathcal{SF}(\widehat{V})$ uniquely determined by the given choice of the fan $\Si\in \mathcal{SF}(V)$ defining $X$.
\end{itemize}}
\end{remark}

\begin{remark}\label{rem:universale}{\rm Given a fan matrix $V$, consider $\widehat{V}=\G(\G(V))$, which is a $CF$--matrix by \cite[Prop.~3.11]{RT-LA&GD}. Let us recall that:
      \begin{equation}\label{universalità}
        \exists\,\b\in\GL_n(\Q)\cap\mathbf{M}_n(\Z)\ :\ V=\b\cdot\widehat{V}
      \end{equation}
      \cite[Prop.~3.1(3)]{RT-LA&GD}. In a sense, the universality property of the 1-covering $Y(\widehat{\Si})$ is summarized by the linear algebraic property (\ref{universalità}). In fact, the proof of Theorem \ref{thm:covering&quotient} given above shows that any further 1-covering of the $\Q$--factorial complete toric variety $X(\Si)$ admitting $V$ as a fan matrix is given by the choice of an integer matrix $\g\in\GL_n(\Q)\cap\mathbf{M}_n(\Z)$ \emph{dividing} the matrix $\b$ defined in (\ref{universalità}), i.e. such that $\b\cdot\g^{-1}\in\GL_n(\Q)\cap\mathbf{M}_n(\Z)$.
Namely:
\begin{itemize}
  \item under our hypothesis, the fan $\Si\in\SF(V)$ is determined as the fan of all the faces of every cone in $\Si(n)$,
  \item $\Si(n)$ is assigned by the fan matrix $V$ and the following subset of the power set $\mathfrak{P}(\{1,\ldots,n+r\})$:\quad\quad\quad$\mathcal{I}_{\Si}:=\{I\subseteq\{1,\ldots,n+r\}\,|\,\langle V_I\rangle\in\Si(n)\}$\,,
  \item consider the set of \emph{divisors of $\b$}:  $$\mathfrak{D}(\b):=\{\eta\in\GL_n(\Q)\cap\mathbf{M}_n(\Z)\,|\,\b\cdot\eta^{-1}\in\GL_n(\Q)
      \cap\mathbf{M}_n(\Z)\}\,,$$
then, for every $\eta\in\mathfrak{D}(\b)$, $\mathcal{I}_{\Si}$ determines a fan $\Si_{\eta}\in\SF(V_{\eta})$, with $V_{\eta}:=\eta\cdot \widehat{V}$, giving a $\Q$--factorial complete toric variety $X_{\eta}(\Si_{\eta})$.
\end{itemize}
In particular $Y(\widehat{\Si})=X_{I_n}(\Si_{I_n})$ and $X(\Si)=X_{\b}(\Si_{\b})$.}
\end{remark}
Then we get the following
\begin{corollary}\label{cor:universale}(Universal 1--covering Theorem) In the notation introduced in Remark \ref{rem:universale}, the following facts hold:
\begin{enumerate}
  \item $Y=X_{I_n}$ is the universal 1--covering PWS of $X$ given by Theorem \ref{thm:covering&quotient}, meaning that $X$ is a finite abelian quotient of $Y$ by the action of $\pi^1_1(X)$,
  \item $X_{\eta}\to X$ is a 1--covering and there is a natural factorization of 1-coverings
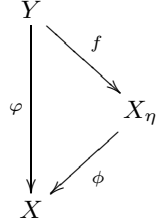
\begin{equation*}
    \xymatrix{Y\ar[dd]_-{\varphi}\ar[dr]^-f&\\
    &X_{\eta}\ar[dl]^-{\phi}\\
    X&}
\end{equation*}
as in Definition \ref{def:1-covering}; in particular $X$ is a finite abelian quotient of $X_{\eta}$ by the action of $\pi^1_1(X)/\pi^1_1(X_{\eta})$.
\end{enumerate}
Consequently, writing $\eta\sim\eta'$ if they are related by means of left multiplication of an element in $\GL_n(\Z)$, the quotient set $\mathfrak{D}(\b)/_{\sim}$ parameterizes all the topologically distinct 1--coverings of $X$.
\end{corollary}

\section{Weil versus Cartier}\label{sez:W&C}

The aim of the present section is that of extending \cite[Thm.~2.9]{RT-LA&GD} to every $\Q$--factorial complete toric variety. By Theorem \ref{thm:covering&quotient} this means extending this result, currently holding for PWS, to their quotients.

Let us start by the following preliminary

\begin{proposition}\label{prop:CartierYX} Let $X(\Si)$ be a $Q$--factorial complete toric variety and $Y(\widehat{\Si})$ be its universal 1-covering. Let $\{D_{\rho}\}_{\rho\in\Si(1)}$ and $\{\widehat{D}_{\rho}\}_{\rho\in\widehat{\Si}(1)}$ be the standard bases of $\Weil(X)$ and $\Weil(Y)$, respectively, given by the torus orbit closures of the rays. Then
\begin{equation*}
     D=\sum_{\rho\in\Si(1)} a_{\rho} D_{\rho}\in\Cart(X)\quad \Longrightarrow\quad \widehat{D}=\sum_{\rho\in\widehat{\Si}(1)} a_{\rho} \widehat{D}_{\rho}\in\Cart(Y)\,.
\end{equation*}
Therefore, under the identification $\Z^{|\Si(1)|}\cong\Weil(X)\stackrel{\a}{\cong}\Weil(Y)\cong\Z^{|\widehat{\Si}(1)|}$ realized by the isomorphism $D_{\rho}\stackrel{\a}{\mapsto}\widehat{D}_{\rho}$,
$$
\Cart(X)\cong\a(\Cart(X))\leq\Cart(Y)\leq\Weil(Y)
$$
is a chain of subgroup inclusions. Moreover the induced morphism $\overline{\a}:\Cl(X)\to\Cl(Y)$ is injective when restricted to $Pic(X)$, realizing the following further chain of subgroup inclusions
$$\Pic(X)\cong\overline{\a}(\Pic(X))\leq\Pic(Y)\leq\Cl(Y)$$.
\end{proposition}

\begin{proof} Let us fix a basis $\mathcal{B}$ of the $\Z$-module $M\cong\Z^n$ and let $V$ and $\widehat{V}$ be fan matrices representing the standard morphisms
$$
div_X:M\cong\Z^n \stackrel{V^T}\longrightarrow\Z^{|\Si(1)|}\cong\Weil(X)\quad,\quad div_Y:M\cong\Z^r \stackrel{\widehat{V}^T}\longrightarrow\Z^{|\widehat{\Si}(1)|}\cong\Weil(Y)
$$
Recalling (\ref{universalità}) in Remark \ref{rem:universale}, let $\b\in\GL_n(\Q)\cap\mathbf{M}_n(\Z)$ be  such that $V=\b\widehat{V}$ and so realizing an injective endomorphism of the $\Z$-module $M$.
The result follows by writing down the condition of being locally principal for a Weil divisor and observing that
\begin{eqnarray}\label{ISigma}
    \mathcal{I}^{\Si}&=&\{I\subseteq\{1,\ldots,n+r\}:\left\langle V^I\right\rangle\in\Si(n)\}\\
    \nonumber
    &=& \{I\subseteq\{1,\ldots,n+r\}:\left\langle \widehat{V}^I\right\rangle\in\widehat{\Si}(n)\}=\mathcal{I}^{\widehat{\Si}}
\end{eqnarray}
by the construction of $\widehat{\Si}\in\SF(\widehat{V})$, given the choice of $\Si\in\SF(V)$. Notice that $\mathcal{I}^\Si$ describes the complements of those sets described by $\mathcal{I}_\Si$, as defined in Remark~\ref{rem:universale}. In particular the Weil divisor $\sum_{j=1}^{n+r}a_jD_j\in\Weil(X)$ is Cartier if and only if
\begin{equation}\label{cartier}
    \forall\,I\in\mathcal{I}^{\Si}\quad\exists\,\mathbf{m}_I\in M : \forall\,j\not\in I\ \v_j^T\mathbf{m}_I=a_j\,,
\end{equation}
where $\v_j$ is the $j$-th column of $V$.
Then $\a(\sum_{j=1}^{n+r}a_jD_j)=\sum_{j=1}^{n+r}a_j\widehat{D}_j$ is a Cartier divisor since
\begin{equation*}
  \forall\,I\in\mathcal{I}^{\Si}\quad \forall\,j\not\in I \quad \widehat{\v}_j^T(\b^T\mathbf{m}_I)=a_j
\end{equation*}
where $\widehat{\v}_j$ is the $j$-th column of $\widehat{V}$. \\
The injectivity of $\overline{\alpha}$ follows from the well-known freeness of $\Pic(X)$.
\end{proof}

We are then in a position of stating and proving the main result of this section.

\begin{theorem}\label{thm:generazione} Let $X=X(\Si)$ be a $n$--dimensional $\Q$--factorial complete toric variety of rank $r$ and $Y=Y(\widehat{\Si})$ be its universal 1--covering. Let $V$ be a reduced fan matrix of $X$, $Q=\G(V)$ a weight matrix of $X$ and $\widehat{V}=\G(Q)$ be a $CF$--matrix giving a fan matrix of $Y$.
\begin{enumerate}
\item Consider the matrix
$$U_Q=(u_{ij})\in\GL_{n+r}(\Z)\ :\ U_Q\cdot Q^T=\HNF\left(Q^T\right)$$
 then the rows of $^rU_Q$ (recall notation in \ref{ssez:lista}) describe the following set of generators of a  \lq\lq free part\rq\rq \  $F\,(\,\cong\Cl(Y))$ of $\Cl(X)$
\begin{equation}\label{generazione}
    \forall\,1\leq i\leq r\quad L_i:=\sum_{j=1}^{n+r} u_{ij} D_j\in \mathcal{W}_T(X)\quad\text{and}\quad F=\bigoplus_{i=1}^r \Z[d_X(L_i)]
\end{equation}
where $d_X:\mathcal{W}_T(X)\to \Cl(X)$ is the morphism giving to a Weil divisor its linear equivalence class.
\item Define $\mathcal{I}^\Si$ as in (\ref{ISigma}). For any $I\in\mathcal{I}^\Si$ let $E_I$ be the $r\times (n+r)$ matrix admitting as rows the standard basis vectors $e_i=(0,\ldots,0,\underset{i}{1},0,\ldots,0)$, for $i\in I$, representing the $i$-th basis divisor $D_i\in\Weil(X)\cong\Z^{|\Si(1)|}$. Set $\widetilde{V}_I:=\left(V^T\,|\,E_I^T\right)\in\mathbf{M}_{n+r}(\Z)$. Then Cartier divisors give rise to the following maximal rank subgroup of $\Weil(X)$
    \begin{equation*}
      \Cart(X)\cong \bigcap_{I\in\mathcal{I}^\Si} \mathcal{L}_c\left(\widetilde{V}_I\right)\leq \Z^{|\Si(1)|}\cong\Weil(X)
    \end{equation*}
  and a basis of $\Cart(X)\leq\Weil(X)$ can be explicitly computed by applying the procedure described in \cite[\S~1.2.3]{RT-LA&GD}.

  \item There exists a choice of the fan matrices $V$ and $\widehat{V}=\G(\G(V))$ and a diagonal matrix $\Delta=\diag(c_1,\ldots,c_n)\in\GL_n(\Q)\cap\mathbf{M}_n(\Z)$ such that
      \begin{itemize}
            \item[(a)] $1=c_1\,|\,\ldots\,|\,c_n$\,,
            \item[(b)] $V=\Delta\cdot\widehat{V}$\,,
            \item[(c)] $\Tors(\Cl(X))\cong\bigoplus_{i=1}^n\Z/c_i\Z=\bigoplus_{k=1}^s\Z/\tau_k\Z$\,,\\ according with the decomposition of $\Cl(X)$ given by the fundamental theorem of finitely generated abelian groups,
         \begin{equation}\label{Chow-decomposizione-tors}
          \Cl(X)=F\oplus\Tors(\Cl(X))\cong\Z^r\oplus\bigoplus_{k=1}^s\Z/\tau_k\Z
         \end{equation}
                where $s< n\,,\,\tau_k=c_{n-s+k}>1\,,\,c_1=\cdots=c_{n-s}=1\,.$
          \end{itemize}

          \item Given the choice of $\widehat{V}=(\widehat{v}_{ij})$ and $V$ as in the previous part (3), then the rows of the submatrix $_s{\widehat{V}}$ describe the following set of generators of $\Tors(\Cl(X))$
            \begin{eqnarray}\label{generazione-tors}
              \forall\,1\leq k\leq s\quad T_k&:=&\sum_{j=1}^{n+r} \widehat{v}_{n-s+k,j} D_j\in \mathcal{W}_T(X)\quad\text{and}\\
              \nonumber
              \Tors(\Cl(X))&=&\mathcal{L}\left(d_X(T_1),\ldots,d_X(T_s)\right)
          \end{eqnarray}
          \item Given the choice of $\widehat{V}$ and $V$ as in the previous parts (3) and (4), consider
              \begin{eqnarray*}
                U&:=&\left(
                     \begin{array}{c}
                       ^rU_Q \\
                       \widehat{V} \\
                     \end{array}
                   \right)\in\GL_{n+r}(\Z)\\
                W &\in&\GL_{n+r}(\Z) \ :\  W\cdot ({^{n+r-s}U})^T=\HNF\left(({^{n+r-s}U})^T\right) \\
                G &:=& {_s\widehat{V}}\cdot\ ({_{s}W})^T \in \mathbf{M}_s(\Z)\\
                U_G&\in&\GL_{s}(\Z) \ :\  U_G\cdot G^T =\HNF(G^T)\,.
              \end{eqnarray*}
              Then a ``torsion matrix''  representing the ``torsion part'' of the morphism $d_X$, that is, $\tau_X:\Weil(X)\to\Tors(\Cl(X))$,
              is given by
              \begin{equation}\label{Gamma}
                  \Ga = {U_G}\cdot\ {_{s}W} \mod {\boldsymbol\tau}
              \end{equation}
              where this notation means that the $(k,j)$--entry of $\Ga$ is given by the class in $\Z/\tau_k\Z$ represented by the corresponding $(k,j)$--entry of ${^sU_G}\cdot\ {_{s}W}$, for every $1\leq k\leq s\,,\,1\leq j\leq n+r$.
\item Let $C_X\in\GL_{n+r}(\Q)\cap\mathbf{M}_{n+r}(\Z)$ be a matrix whose rows give a basis of $\Cart(X)$ in $\Weil(X)$, as obtained in the previous part 2. Identify $\Cl(X)$ with $\Z^r\oplus\bigoplus_{k=1}^s\Z/\tau_k\Z$ by item (c) of part 3, and represent the morphism $d_X$ by $Q\oplus \Gamma$, according to (5). Let $A\in\GL_{n+r}(\Z)$ be a matrix such that $A\cdot C_X \cdot Q^T$ is in $\HNF$. Let $\mathbf{c}_1,\ldots,\mathbf{c}_r$ be the first $r$ rows of the matrix $A\cdot C_X$ and for $i=1,\ldots r$ put $\mathbf{b}_i=Q\cdot\mathbf{c}_i^T + \Ga\cdot \mathbf{c}_i^T$. Then $\mathbf{b}_1,\ldots \mathbf{b}_r$ is a basis of the free group $\Pic(X)$ in $\Cl(X)$.

\item Setting $\d_{\Si}:=\lcm\left(\det(Q_I):I\in\mathcal{I}^\Sigma\right)$
then
$$\d_{\Si}\mathcal{W}_T(X)\subseteq \mathcal{C}_T(X)\quad\text{and}\quad\d_{\Si}\mathcal{W}_T(Y)\subseteq \mathcal{C}_T(Y)$$
and there are the following divisibility relations
$$\d_{\Si}\ |\ [\Cl(Y):\Pic(Y)]=[\mathcal{W}_T(Y):\mathcal{C}_T(Y)]\ |\ [\Cl(X):\Pic(X)]= [\mathcal{W}_T(X):\mathcal{C}_T(X)]\,.$$
\end{enumerate}
\end{theorem}

\begin{proof}
(1): Recall that $Q$ is a representative matrix of the ``free part'' of the morphism $d_X:\Weil(X)\to\Cl(X)$, that is, the morphism $f_X:\Weil(X)\to F\cong\Cl(Y)$. Then the proof goes on as proving part (1) of \cite[Thm.~2.9]{RT-LA&GD}: namely
$$ U_Q\cdot Q^T = \HNF\left(Q^T\right)= \left(
                                          \begin{array}{c}
                                            \mathbf{I}_r \\
                                            \mathbf{0}_{n,r} \\
                                          \end{array}
                                        \right)\quad\Longrightarrow\quad Q\cdot\ ^rU_Q = \mathbf{I}_r\,.
$$
(2): Recalling relation (\ref{cartier}) in the proof of Proposition \ref{prop:CartierYX}, set
$$\forall\,I\in\mathcal{I}^{\Si}\quad\mathcal{P}^I=\{L=\sum_{j=1}^{n+r}a_jD_j\in \mathcal{W}_T(X)\ |\ \exists\,\mathbf{m}\in M : \forall\,j\not\in I\ \mathbf{m}\cdot\v_j=a_j\}.$$
Then $\mathcal{P}^I$ contains $\mathrm{Im}(div_X:M\to\Weil(X))=\mathcal{L}_c\left(V^T\right)$ and a $\Z$-basis of $\mathcal{P}^I$ is given by
$$\{D_j, j\in I\}\cup\{\sum_{k=1}^{n+r}v_{ik}D_k, i=1,\ldots ,n\},$$
where $\{v_{ik}\}$ is the $i$-th entry of $\mathbf{v}_k$, so giving the rows of the matrix $\widetilde{V}_I$ defined in the statement.

 \noindent (3): By (\ref{universalità}) there exists a matrix $\b\in\GL_n(\Q)\cap\mathbf{M}_n(\Z)$ such that $\b\cdot\widehat{V}=V$. Define $\Delta:=\SNF(\b)$. It is then a well known fact (see e.g. \cite[Algorithm~2.4.14]{Cohen}) the existence of matrices $\mu,\nu\in\GL_n(\Z)$ such that
 \begin{equation}\label{betaSmith}
 \Delta=\mu\cdot\b\cdot\nu\quad\Longrightarrow\quad \Delta\cdot(\nu^{-1}\cdot\widehat{V})=\mu\cdot V\,.
 \end{equation}
 Notice that $\nu^{-1}\cdot\widehat{V}\sim\widehat{V}$ are equivalent $CF$--matrices and $\mu\cdot V\sim V$ are equivalent $F$--matrices giving a choice of the fan matrices of $Y$ and $X$, respectively, satisfying conditions (a) and (b): notice that $c_1=\gcd(c_1,\ldots,c_n)=1$ since $V$ is a reduced $F$--matrix. Then (c) follows by recalling that
 \begin{equation*}
   \Tors(\Cl(X))\cong \Z^n/ \mathcal{L}_r(T_n)\,,
\end{equation*}
where $T_n$ is the upper $n\times n$ submatrix of $\HNF(V^T)$ (see relations (10) in \cite[Thm.~2.4]{RT-LA&GD}).

 \noindent (4): If $\widehat{V}$ and $V$ are such that condition (b) in (4) holds, then the $k$--th row of $_s\widehat{V}$ is actually the $(n-s+k)$--th row of $V$ divided by the $\gcd$ of its entries, i.e. by $c_{n-s+k}=\tau_k>1$ . Since the rows of $V$ span $\ker(d_X)$, the $k$--th row of $_s\widehat{V}$ gives then rise to the torsion Weil divisor $T_k=\sum_{j=1}^{n+r} \widehat{v}_{n-s+k,j} D_j$. Recalling (4)(c) the classes $d_X(T_1),\ldots,d_X(T_s)$ suffice to generate $\Tors(\Cl(X))$.

 \noindent (5): A representative matrix of the torsion part $\tau_X:\Weil(X)\to\Cl(X)$ of the morphism $d_X$ in diagram (\ref{div-diagram-covering}) is any matrix satisfying the following properties:
  \begin{itemize}
    \item[$(i)$] $\Ga=(\g_{kj})$ with $\g_{kj}\in\Z/\tau_k\Z$,
    \item[$(ii)$] $\Ga\cdot (^rU_Q)^T=\mathbf{0}_{s,r} \mod \boldsymbol\tau$, meaning that $\Ga$ kills the generators of the free part $F\leq\Cl(X)$ defined in (\ref{generazione}),
    \item[$(iii)$] $\Ga\cdot V^T=\mathbf{0}_{s,n} \mod \boldsymbol\tau$, where $V$ is a fan matrix satisfying condition 4.(b): this is due to the fact that the rows of $V$ span $\ker(d_X)$,
    \item[$(iv)$] $\Ga\cdot({_s\widehat{V}})^T=\mathbf{I}_s \mod \boldsymbol\tau$, since the rows of ${_s\widehat{V}}$ give the generators of $\Tors(\Cl(X))$, as in (\ref{generazione-tors}).
  \end{itemize}
  Therefore it suffices to show that the matrix $ {U_G}\cdot\ {_{s}W}$ in (\ref{Gamma}) satisfies the previous conditions $(ii)$, $(iii)$ and $(iv)$ without any reduction mod $\boldsymbol\tau$, that is,
  \begin{equation*}
     {U_G}\cdot\ {_{s}W}\cdot\ ({^{n+r-s}U})^T=\mathbf{0}_{s,n+r-s}\quad,\quad {U_G}\cdot\ {_{s}W}\cdot\ ({_s\widehat{V}})^T=\mathbf{I}_s\,.
  \end{equation*}
 The first equation follows by the definition of $W$, in fact
 \begin{equation*}
     W\cdot ({^{n+r-s}U})^T=\HNF\left(({^{n+r-s}U})^T\right)
     =\left(
                                                                                              \begin{array}{c}
                                                                                                \mathbf{I}_{n+r-s} \\
                                                                                                \mathbf{0}_{s,n+r-s} \\
                                                                                              \end{array}
                                                                                            \right)\,\Rightarrow\,{_{s}W}\cdot\ ({^{n+r-s}U})^T=\mathbf{0}_{s,n+r-s}
 \end{equation*}
 The second equation follows by the definition of $U_G$, in fact
 \begin{equation*}
     U_G\cdot {_{s}W}\cdot\ ({_s\widehat{V}})^T= U_G\cdot G^T =\HNF(G^T)= \mathbf{I}_s\,.
 \end{equation*}

\noindent (6): By definition $$\Pic(X)=\mathrm{Im}(\Cart(X)\hookrightarrow\Weil(X)\stackrel{d_X}{\to}\Cl(X))$$
so that $\Pic(X)$ is generated by the image under $Q\oplus \Gamma$ of the transposed of the rows of $C_X$. Since $\rk(C_X)=n+r$ and $\rk(Q)=r$, the matrix $C_X\cdot Q^T$ has rank $r$ and therefore its $\HNF$ has the last $n-r$ rows equal to zero. Therefore the rows of the matrix $A\cdot C_X$ provide a basis of $\Cart(X)$ in $\Weil(X)$ such that its last $n$ rows are a basis of $\mathcal{L}_r(\widehat V)\cap \Cart(X)=\mathcal{L}_r( V)$. Since $\Pic(X)$ is free of rank $r$ it is freely generated by the images under $d_X$ of the first $r$ rows.

 \noindent (7): Part (4) of \cite[Thm.~2.9]{RT-LA&GD} gives that $\d_{\Si}\ |\  [\Cl(Y):\Pic(Y)]=[\Weil(Y):\Cart(Y)]$. On the other hand Proposition \ref{prop:CartierYX} gives that $[\mathcal{W}_T(Y):\mathcal{C}_T(Y)]\ |\ [\mathcal{W}_T(X):\mathcal{C}_T(X)]=[\Cl(X):\Pic(X)]$.
\end{proof}

\begin{remark}\label{rem:} {\rm This is the generalization of \cite[Rem.~2.10]{RT-LA&GD}. The most part of results stated in Theorem~\ref{thm:generazione} are based on the algorithms giving the $\HNF$ and the $\SNF$ of a matrix, and associated switching matrices, which are well known algorithms (see e.g. \cite[Algorithms 2.4.4 and 2.4.14]{Cohen}) implemented in many computer algebra procedures. Then Theorem \ref{thm:generazione} gives quite effective and constructive me\-thods to produce a lot of interesting information characterizing a given $\Q$--factorial complete toric variety $X$. In particular
\begin{itemize}
  \item[i.] given a fan matrix $V$ of $X$, a switching matrix $U_V\in\GL_{n+r}(\Z)$, such that $\HNF(V^T)=U_V\cdot V^T$, encodes a weight matrix $Q=\G(V)$ of $X$ in the last $r$ rows \cite[Prop.~4.3]{RT-LA&GD};
  \item[ii.] given a weight matrix $Q$ of $X$, a switching matrix $U_Q\in\GL_{n+r}(\Z)$, such that $\HNF(Q^T)=U_Q\cdot Q^T$, encodes both a fan matrix $\widehat{V}$ of the universal 1-covering $Y$ of $X$, given by the lower $n$ rows of $U_Q$ \cite[Prop.~4.3]{RT-LA&GD}, and a basis $\{d_X(L_i)\}_{i=1}^r$ of a free part $F\cong\Cl(Y)\cong\Z^r$ of the divisor class group $\Cl(X)$, given by the upper $r$ rows of $U_Q$, as in part 1 of Theorem \ref{thm:generazione};
  \item[iii.] by (\ref{universalità}) there exists $\b\in\GL_n(\Z)$ such that $\b\cdot\widehat{V}=V$; an effective procedure producing $\b$ is the following:
      \begin{itemize}
        \item[-] set
        \begin{eqnarray*}
          H=(h_{i,j}):=\HNF(V) &,& \widehat{H}=(\hat{h}_{i,j}):=\HNF(\widehat{V}) \\
          U\in\GL(n,\Z) &:& U\cdot V=H \\
          \widehat{U}\in\GL(n,\Z) &:& \widehat{U}\cdot \widehat{V}=\widehat{H} \\
          \b_H=(b_{i,j}) &:=& U\cdot\b\cdot\widehat{U}^{-1}\,,
        \end{eqnarray*}
        \item[-] then $V=\b\cdot \widehat{V}\ \Rightarrow\ H=\b_H\cdot\widehat{H}$\,: since, up to a simultaneous permutations of columns of $V$ and $\widehat{V}$, one can assume that both the submatrices $H_{\{1,\ldots,n\}}$ and $\widehat{H}_{\{1,\ldots,n\}}$ are upper triangular, then also $\b_H$ turns out to be upper triangular,
        \item[-] one can then get $\b_H$ from $H$ and $\widehat{H}$, by the following recursive relations
        \begin{eqnarray}\label{ricorsive}
          \forall\,1\leq i\leq n \quad b_{i,i}&=&\frac{h_{i,i}}{\hat{h}_{i,i}} \\
          \nonumber
          \forall\,i+1\leq j\leq n \quad b_{i,j}&=&\frac{1}{\hat{h}_{i,i}}\left(h_{i,j}-\sum_{k=i}^{j-1}b_{i,k}\hat{h}_{k,j}\right)
        \end{eqnarray}
              \end{itemize}
  \item[iv.] let $\Delta:=\SNF(\b)$ and $\mu,\nu\in\GL_n(\Z)$ be matrices realizing relations (\ref{betaSmith}); setting $V':=\nu^{-1}\cdot\widehat{V}$, construct the submatrix $_sV'$ as in part 4 of Theorem \ref{thm:generazione}; then the rows of $_sV'$ describe a set of generators $\{d_X(T_k)\}_{k=1}^s$ of $\Tors(\Cl(X))$, as in (\ref{generazione-tors});
  \item[v.] get matrices $W$, $G$ and $U_G$ as in part 5 of Theorem \ref{thm:generazione} and construct the torsion matrix $\Ga$ as in (\ref{Gamma}), giving a representative matrix of the torsion part $\tau_X$ of the morphism $d_X$;
  \item[vi.] apply procedure \cite[\S~1.2.3]{RT-LA&GD}, also based on the $\HNF$ algorithm, to get a $(n+r)\times (n+r)$ matrix $C_X$ whose rows give a basis of $\Cart(X)\leq\Weil(X)\cong\Z^{|\Si(1)|}$;
  \item[vii.]  apply procedure described in part 6 of Theorem \ref{thm:generazione} to get a system of  generators of $\Pic(X)$ in $ \Cl(X)\,.$ Precisely, let $A\in\GL_{n+r}(\Z)$ be a switching matrix such that $\HNF(C_X\cdot Q^T)=A\cdot C_X \cdot Q^T$, and put
  \begin{equation}
  \label{eq:BX}
  B_X=\ ^r(A\cdot C_X \cdot Q^T),\quad \Theta_X=\ ^r{(A\cdot C_X \cdot \Gamma^T)}\end{equation} \\
 Then the rows of the matrices $B_X$ and $\Theta_X$ represent  respectively the free part and the torsion part of a basis of $\Pic(X)$ in $\Cl(X)$, where the latter is identified to $ \Z^r\oplus\bigoplus_{k=1}^s\Z/\tau_k\Z$ by (\ref{Chow-decomposizione-tors}).

Moreover:
\begin{itemize}
  \item recall that, for the universal 1--covering $Y$ of $X$, once fixed the basis $\{\widehat{D}_j\}_{j=1}^{n+r}$ of $\mathcal{W}_T(Y)\cong\Z^{n+r}$ and the basis $\{d_Y(\widehat{L}_i)\}_{i=1}^r$ of $\Cl(Y)\cong\Z^r$, constructed as in (\ref{generazione}) by replacing $D_j$ with $\widehat{D}_j$ (see (11) in \cite[Thm.~2.9]{RT-LA&GD}), one gets  the following commutative diagram
  \begin{equation*}
\def\objectstyle{\displaystyle}
\xymatrix{
& 0 \ar[d] && 0 \ar[d] && 0 \ar[d] & \\
0 \ar[r] & M \ar[rr]^-{\left(
                                         \begin{array}{c}
                                           \mathbf{0}_{n,r}\,|\,\mathbf{I}_n  \\
                                         \end{array}
                                       \right)}\ar@{=}[d] &&
\mathcal{C}_T(Y)\cong\Pic(Y)\oplus M \ar[rr]^-{\left(
                                         \begin{array}{c}
                                           \mathbf{I}_r\,|\,\mathbf{0}_{r,n}  \\
                                         \end{array}
                                       \right)}\ar[d]^-{C_Y^T} && {\Pic(Y)} \ar[r]\ar[d]^-{B_Y^T} & 0 \\
0 \ar[r] & M \ar[rr]^-{div_Y}_-{\widehat{V}^T}\ar[d] && \mathcal{W}_T(Y)=\bigoplus_{j=1}^{n+r} \Z \cdot D_{j}\ar[d]
\ar[rr]^-{d_Y}_-Q && \Cl(Y)\ar[d] \ar[r] & 0 \\
 & 0\ar[rr] && \mathcal{T}_Y\ar[d]\ar[rr]^-{\cong} && \mathcal{T}_Y\ar[r]\ar[d]&0\\
 & &&0&&0& }
\end{equation*}
where $B_Y$ is the $r\times r$ matrix constructed in \cite[Thm.~2.9(3)]{RT-LA&GD} and
\begin{equation*}
    C_Y=\begin{pmatrix}B_Y & \mathbf{0}_{r,n}\\  \mathbf{0}_{n,r}& \mathbf{I}_{n}\end{pmatrix}\cdot U_Q= \begin{pmatrix}B_Y\cdot\,^rU_Q\\  \widehat{V}\end{pmatrix}\,,
\end{equation*}
\item once fixed the basis $\{D_j\}_{j=1}^{n+r}$ for $\mathcal{W}_T(X)\cong\Z^{n+r}$ and the basis $\{d_X(L_i)\}_{i=1}^r$ of the free part $F\cong\Z^r$ of $\Cl(X)$, constructed in (\ref{generazione}), one gets  the following commutative diagram
  \begin{equation*}
\def\objectstyle{\displaystyle}
\xymatrix{
& 0 \ar[d] && 0 \ar[d] && 0 \ar[d] & \\
0 \ar[r] & M \ar[rr]^-{\left(
                                         \begin{array}{c}
                                           \mathbf{0}_{n,r}\,|\,\mathbf{I}_n  \\
                                         \end{array}
                                       \right)}\ar@{=}[d] &&
\mathcal{C}_T(X)\cong\Pic(X)\oplus M \ar[rr]^-{\left(
                                         \begin{array}{c}
                                           \mathbf{I}_r\,|\,\mathbf{0}_{r,n}  \\
                                         \end{array}
                                       \right)}\ar[d]^-{C_X^T} && {\Pic(X)} \ar[r]\ar[d]^-{B_X^T\oplus \Theta_X^T} & 0 \\
0 \ar[r] & M \ar[rr]^-{div_X}_-{V^T}\ar[d] && \mathcal{W}_T(X)=\bigoplus_{j=1}^{n+r} \Z \cdot D_{j}\ar[d]
\ar[rr]^-{d_X=f_X\oplus\tau_X}_-{Q\oplus\Ga} && \Cl(X)\ar[d] \ar[r] & 0 \\
 & 0\ar[rr] && \mathcal{T}_X\ar[d]\ar[rr]^-{\cong} && \mathcal{T}_X\ar[r]\ar[d]&0\\
 & &&0&&0& }
\end{equation*}

\end{itemize}

Moreover:
\begin{itemize}
\item recall the following commutative diagram of short exact sequences
  \begin{equation}\label{div-diagram-covering}
    \begin{array}{c}
      \xymatrix{&&&0\ar[d]&\\
& 0 \ar[d] & 0 \ar[d] & \ker(\overline{\a})=\Tors(\Cl(X)) \ar[d] & \\
0 \ar[r] & M \ar[r]^-{div_X}_-{V^T}\ar[d]_-{\b^T} &
\mathcal{W}_T (X)=\Z^{|\Si(1)|} \ar[r]^-{d_X}\ar[d]^-{\a}_-{\mathbf{I}_{n+r}} & \Cl(X) \ar[r]\ar[d]^-{\overline{\a}} & 0 \\
0 \ar[r] & M \ar[r]^-{div_Y}_-{\widehat{V}^T}\ar[d]&\mathcal{W}_T(Y)=\Z^{|\widehat{\Si}(1)|}\ar[r]^-{d_Y}\ar[d] & \Cl (Y) \ar[r]\ar[d] & 0 \\
 & \coker(\b^T)\cong\Tors(\Cl(X))\ar[d] & 0 & 0 & \\
 &0&&&}
    \end{array}
\end{equation}
\end{itemize}
then, putting all together, one gets the following 3--dimensional commutative diagram
\begin{equation}\label{diagramma3D}
\begin{array}{c}
  \xymatrix{M\ar@{=}[dddd]\ar@{^{(}->}[rrr]^-{div_X}_-{\left(
                                         \begin{array}{c}
                                           \mathbf{0}_{n,r}\,|\,\mathbf{I}_n  \\
                                         \end{array}
                                       \right)}\ar@{^{(}->}[dr]^-{\b^T}&&&\Cart(X)\ar@{^{(}->}[dr]^-{\a_|}_>>>>>>{(C_X\cdot C_Y^{-1})^T}\ar@{->>}[rrr]^-{{d_X}_|}_-{\left(
                                         \begin{array}{c}
                                           \mathbf{I}_r\,|\,\mathbf{0}_{r,n}  \\
                                         \end{array}
                                       \right)}\ar@{^{(}->}[dddd]_>>>>>>>>>>>>>{C_X^T}&&&
  \Pic(X)\ar@{^{(}->}[dr]^-{\overline{\a}_|}_>>>>>>{(B_X\cdot B_Y^{-1})^T}\ar@{^{(}->}[dddd]\ar@{^{(}->}[dddd]_>>>>>>>>>>>>>>>{B_X^T\oplus\Theta_X^T}&&&&\\
             &M\ar@{^{(}->}[rrr]^-{div_Y}_-{\left(
                                         \begin{array}{c}
                                           \mathbf{0}_{n,r}\,|\,\mathbf{I}_n  \\
                                         \end{array}
                                       \right)}\ar@{=}[dddd]\ar@{->>}[dr]&&&\Cart(Y)\ar@{->>}[rrr]^-{{d_Y}_|}_-{\left(
                                         \begin{array}{c}
                                           \mathbf{I}_r\,|\,\mathbf{0}_{r,n}  \\
                                         \end{array}
                                       \right)}\ar@{^{(}->}[dddd]_-{C_Y^T}\ar@{->>}[dr]&&&\Pic(Y)\ar@{->>}[dr]\ar@{^{(}->}[dddd]_-{B_Y^T}&&\\
             &&\coker(\b^T)\ar@{=}[dddd]\ar@{^{(}->}[rrr]&&&\coker(\a_|)\ar@{->>}[rrr]&&&\coker(\overline{\a}_|)\\
             &&&&&\ker(\overline{\a})\ar@{^{(}->}[dddd]\ar@{^{(}->}[dr]&&&&\\
             M\ar@{^{(}->}[rrr]^-{div_X}_-{V^T}\ar@{^{(}->}[dr]^-{\b^T}&&&\Weil(X)\ar@{->>}[dddd]\ar@{->>}[rrr]^-{d_X=f_X\oplus\tau_X}_-{Q\oplus\Ga}\ar[dr]^-{\a}_{\mathbf{I}_{n+r}}&&&\Cl(X)\ar@{->>}[dr]^-{\overline{\a}}_-{\mathbf{I}_r\oplus\mathbf{0}_r}\ar@{->>}[dddd]&&&&\\
             &M\ar@{^{(}->}[rrr]^-{div_Y}_-{\widehat{V}^T}\ar@{->>}[dr]&&&\Weil(Y)\ar@{->>}[dddd]\ar@{->>}[rrr]^-{d_Y}_-{Q}&&&\Cl(Y)\ar@{->>}[dddd]&&\\
             &&\coker(\b^T)&&&&&&\\
             &&\mathcal{K}\ar@{^{(}->}[dr]\ar[rrr]_-{\cong}&&&\mathcal{K}\ar@{^{(}->}[dr]&&\\
             &&&\mathcal{T}_X\ar@{->>}[dr]\ar[rrr]_-{\cong}&&&\mathcal{T}_X\ar[dr]&&\\
             &&&&\mathcal{T}_Y\ar[rrr]_-{\cong}&&&\mathcal{T}_Y}
\end{array}
\end{equation}
 The Snake Lemma implies
\begin{eqnarray*}
  \coker(\b^T)&\cong&\ker(\overline{\a})\cong\Tors(\Cl(X))\\
  \mathcal{K}&\cong&\coker(\a_|)\cong\Cart(Y)/\Cart(X)
\end{eqnarray*}
so giving the following short exact sequences on torsion subgroups
\begin{equation}\label{torsione}
  \xymatrix{&&0\ar[d]&\\
            0\ar[r]&\Tors(\Cl(X))\ar[r]&\Cart(Y)/\Cart(X)\ar[r]\ar[d]&\Pic(Y)/\Pic(X)\ar[r]&0\\
            &&\Cl(X)/\Pic(X)\ar[d]&\\
            &&\Cl(Y)/\Pic(Y)\ar[d]&\\
            &&0&}
\end{equation}
\end{itemize}}
\end{remark}

\section{Going back: from the geometric quotient to the fan}\label{sez:fan}

In the present section we want to reverse our point of view. Namely assume that a $\Q$--factorial complete $n$--dimensional toric variety $X$ of rank $r$ is presented as a geometric quotient as follows:
\begin{itemize}
  \item let $Z$ be an algebraic subset of $\C^{n+r}$ defined by a suitable monomial ideal $B\subseteq\C[x_1,\ldots,x_{n+r}]$,
  \item consider the reductive subgroup $G=\Hom(A,\C^*)\subseteq(\C^*)^{n+r}$ where $A$ is a finitely generated abelian group of rank $r$; this means that
      $$G\cong(\C^*)^r\oplus\bigoplus_{k=1}^s\mu_{\tau_k}$$
      where $\mu_{\tau_k}$ is the cyclic group of $\tau_k$--th roots of unity with $1<\tau_1\,|\,\cdots\,|\,\tau_s$;
  \item assume that the action of $G$ over $\C^{n+r}$ is equivariant with respect to the usual multiplication of $(\C^*)^{n+r}$, meaning that for every $g=(\t,\boldsymbol\varepsilon)\in G$, with $\t=(t_1,\ldots,t_r)\in(\C^*)^r$ and $\boldsymbol\varepsilon=(\varepsilon_1,\ldots,\varepsilon_s)\in \bigoplus_{k=1}^s\mu_{\tau_k}$, there exist matrices
      \begin{equation}\label{matrici}
        Q=(q_{ij})\in \mathbf{M}(r,n+r;\Z)\quad,\quad C=(c_{kj})\in\mathbf{M}(s,n+r;\Z)
      \end{equation}
      such that the action is given by the usual multiplication as follows
     \begin{equation}\label{moltiplicazione}
        \forall\,\x\in\C^{n+r}\quad(\t,\boldsymbol\varepsilon)\cdot \x = \left(\prod_{i=1}^r t_{i}^{q_{ij}}\cdot\prod_{k_=1}^s
        \varepsilon_k^{a_k c_{kj}}\cdot x_j\right)_{1\leq j\leq n+r}
     \end{equation}
  \item assume that the given action $G\times\C^{n+r}\to\C^{n+r}$ defines a geometric quotient $(\C^{n+r}\backslash Z)/G$ giving precisely the toric variety $X$; this is possible for every $\Q$--factorial complete toric variety by the well known Cox' result \cite{Cox}.
\end{itemize}

\begin{remark}\label{rem:dati-quot} {\rm By (\ref{matrici}) and (\ref{moltiplicazione}) $X$ is completely assigned by the following three data:
\begin{itemize}
  \item[i.] a reduced $W$--matrix $Q=(q_{ij})$ defining the action of the free part $(\C^*)^r$ of $G$: the fact that $Q$ can be assumed to be a $W$--matrix follows by the completeness of $X$ \cite[Thm.~3.8]{RT-LA&GD}; moreover a $W$--matrix and its reduction define isomorphic actions of the free part of $G$;
  \item[ii.] the \emph{torsion} matrix $\Ga=(\g_{kj})$, with $\g_{kj}=[c_{k,j}]_{\tau_k}\in\Z/\tau_k\Z$, defining the action of the torsion part of $G$: in fact such an action is invariant with respect to the choice of different representatives of the class $\g_{k,j}$;
  \item[iii.] the algebraic subset $Z\subseteq\C^{n+r}$.
\end{itemize}
\begin{description}
  \item[Wanted] \emph{a fan matrix $V$ of $X$ and a fan $\Si\in\SF(V)$ defining $X$.}
\end{description}
Let us first of all notice that Theorem \ref{thm:covering&quotient} gives immediately the PWS $Y$ which is the universal 1--covering of $X$:  this is obtained by Gale duality since $\widehat{V}=\G(Q)$ is a $CF$--matrix (by \cite[Prop.~3.12(1)]{RT-LA&GD}), hence giving a fan matrix of $Y$. Moreover the \emph{irrelevant ideal} $B$ defining $Z$ reconstructs the $n$--skeleton of a fan $\widehat{\Si}\in\SF(\widehat{V})$ obtained as the fan of all the faces of every cone in $\widehat{\Si}(n)$.

Therefore to get $V$ and a fan $\Si$ of $X$ it suffices to recover a matrix
\begin{equation*}
    \b\in\GL_n(\Q)\cap\mathbf{M}_n(\Z)\ :\ \b\cdot\widehat{V}=V
\end{equation*}
as in (\ref{universalità}). Then the fan $\Si\in\SF(V)$ is obtained from $\widehat{\Si}\in\SF(\widehat{V})$, by reversing the process described in the proof of Theorem \ref{thm:covering&quotient}, that is, by means of the subset $I_{\widehat{\Si}}\subseteq\mathfrak{P}(\{1,\ldots,n+r\})$, as defined in Remark \ref{rem:universale}.}
\end{remark}

\begin{theorem}\label{thm:da-quot-a-fan} Given a reduced $W$--matrix $Q$ and a torsion matrix $\Ga$, as in parts \emph{i} and \emph{ii} of Remark \ref{rem:dati-quot}, describing a $Q$--factorial complete toric variety $X$ as a geometric quotient, and setting $\widehat{V}=\G(Q)$, then a matrix $\b\in\GL_n(\Q)\cap\mathbf{M}_n(\Z)$ such that $V=\b\cdot\widehat{V}$ is a fan matrix of $X$ can be reconstructed as follows:
\begin{enumerate}
  \item define a matrix $C\in\mathbf{M}(s,n+r;\Z)$ by choosing a representative $c_{kj}$ for any entry $\g_{kj}$ of $\Ga$,
  \item consider the matrix $K:=\left(
                                  \begin{array}{ccc}
                                    &\widehat{V}\cdot C^T& \\
                                    \tau_1&\cdots&0\\
                                    \vdots&\ddots&\vdots\\
                                    0&\cdots&\tau_s\\
                                  \end{array}
                                \right)\in \mathbf{M}(n+s,s;\Z)$
  \item let $U\in\GL_{n+s}(\Z)\,:\,\HNF(K)=U\cdot K$
\end{enumerate}
then, recalling notation in \ref{ssez:lista},
$$\b=({_nU})_{\{\1,\ldots,n\}}$$
\end{theorem}

\begin{proof} A fan matrix $V=(v_{ij})$ of $X$ has to satisfy the relations
\begin{equation}\label{relazioni}
    V\cdot Q^T=0\quad,\quad V\cdot\Ga^T= 0 \mod \boldsymbol\tau
\end{equation}
where the latter means that $\sum_{j=1}^{n+r} v_{ij}\g_{kj}=[0]_{\tau_k}$, for every $1\leq i\leq n\,,\,1\leq k\leq s$. Since by (\ref{universalità}) there exists a matrix $\b$ such that $V=\b\cdot\widehat{V}$, the former relation is a consequence of the fact that $\widehat{V}=\G(Q)$. Setting $\b=(x_{l,i})$ and $\widehat{V}=(\widehat{v}_{ij})$, the second relation in (\ref{relazioni}) can be rewritten as follows
\begin{eqnarray*}
    \forall\,1\leq l\leq n\,,\,1\leq k\leq s&& \sum_{i=1}^n\left(\sum_{j=1}^{n+r} v_{ij}\g_{kj}\right)x_{l,i} = [0]_{\tau_k}\\
    &\Longleftrightarrow& \sum_{i=1}^n\left(\sum_{j=1}^{n+r} v_{ij}c_{kj}\right)x_{l,i}+{\tau_k}(-x_{l+k}) =0
\end{eqnarray*}
where $c_{kj}$ is a representative of $\g_{kj}$. Then we are looking for $n$ independent integer common solutions of a system of $s$ linear homogeneous equations in $n+s$ variables whose matrix is given by $K^T$, where $K$ is defined as in part 2 of the statement. These solutions are given by the lower $n$ rows of a matrix
$$U\in\GL_{n+s}(\Z)\,:\quad U\cdot K=\HNF(K)=\left(
                                           \begin{array}{c}
                                             T_s \\
                                             \mathbf{0}_{n,s} \\
                                           \end{array}
                                         \right)
$$
where $T_s$ is an upper triangular matrix. In particular the first $n$ columns of ${_nU}$ give a $n\times n$ matrix $\b$ such that $$\b\cdot\widehat{V}\cdot\Ga^T=0\mod\boldsymbol\tau$$
ending up the proof.
\end{proof}

\begin{remark}\label{rem:isoquozienti} {\rm The previous Theorem \ref{thm:da-quot-a-fan} gives also a method to check if two geometric quotients giving $\Q$--factorial complete toric varieties are actually isomorphic. In fact by recovering the fan matrices this problem is turned into the problem of checking if the fan matrices are equivalent, up to a permutation on columns, in the sense of \cite[\S~3]{RT-LA&GD}, and if their respective fans are each other related by such a fan matrices equivalence. If only the latter fact is not true then the two given toric geometric quotients are still birationally isomorphic (and isomorphic in codimension 1 \cite[\S~15.3]{CLS}).}
\end{remark}

\section{Examples}\label{sez:Esempi} In this section we are going to present concrete applications of Theorems \ref{thm:covering&quotient}, \ref{thm:generazione} and \ref{thm:da-quot-a-fan}.

\begin{example}\label{ex:K2} {\rm The following example is the same given in \cite[Ex.~1.3]{Kasprzyk}, which is here re-discussed to introduce the reader to the use of all the above illustrated techniques in the easier case of an already known example with $r=1$.

Let us consider the 3-dimensional toric variety $X=X(\Si)$, $\Si=\fan(\v_1,\v_2,\v_3,\v_4)$ and $\v_i$ is the $i$-th column of the $3\times 4$ fan matrix
\begin{equation}\label{KPRZ-fan}
    V=\left(
      \begin{array}{cccc}
        1 & 0 & 1 & -2 \\
        0 & 1 & -3 & 2 \\
        0 & 0 & 5 & -5 \\
      \end{array}
    \right)\ .
\end{equation}
Let us follows the list from i. to vii. in Remark \ref{rem:}.

\noindent i. A matrix $U_V\in\GL_4(\Z)$ such that $\HNF(V^T)=U_V\cdot V^T$ is given by
\begin{equation*}
    U_V=\left(
          \begin{array}{cccc}
            1 & 0 & 0 & 0 \\
            0 & 1 & 0 & 0 \\
            -1 & 3 & 1 & 0 \\
            1 & 1 & 1 & 1 \\
          \end{array}
        \right)
\end{equation*}
whose last row gives the weight matrix $Q=\left(
                                            \begin{array}{cccc}
                                              1 & 1 & 1 & 1 \\
                                            \end{array}
                                          \right)
$ of $X$: then the universal 1--covering $Y$ of $X$ is the projective space $Y=\P^3$.

\noindent ii. A switching matrix $U_Q\in\GL_4(\Z)$ such that $\HNF(Q^T)=U_Q\cdot Q^T$ is given by
\begin{equation*}
    U_Q=\left(
          \begin{array}{cccc}
            1 & 0 & 0 & 0 \\
            1 & 0 & 1 & -2 \\
            0 & 1 & -3 & 2 \\
            0 & 0 & 1 & -1 \\
          \end{array}
        \right)
\end{equation*}
encoding a basis $\{d_X(L)\}$ of the free part $F\cong\Z$ of $\Cl(X)$, given by the first row and meaning that $L=D_1$, and a fan matrix $\widehat{V}$ of the universal 1--covering $Y=\P^3$, given by the further rows, that is,
\begin{equation*}
   \widehat{V}=\left(
          \begin{array}{cccc}
            1 & 0 & 1 & -2 \\
            0 & 1 & -3 & 2 \\
            0 & 0 & 1 & -1 \\
          \end{array}
        \right)\sim \left(\begin{array}{cccc}
            1 & 0 & 0 & -1 \\
            0 & 1 & 0 & -1 \\
            0 & 0 & 1 & -1 \\
          \end{array}
        \right)\,.
\end{equation*}
iii. One can then immediately get a matrix $\b$ such that $\b\cdot \widehat{V}=V$, that is, $\b=\diag(1,1,5)$ and already in $\SNF$.

\noindent iv. Therefore $s=1$ and the last row of $\widehat{V}$ gives a generator of
$$\Tors(\Cl(X))\cong \Z/5\Z\,,$$
namely $d_X(T)$ with
$T=D_3-D_4$. Hence
\begin{equation*}
    \Cl(X)=\Z\cdot  f_X(D_1)\oplus \Z \cdot \tau_X(D_3-D_4)\cong \Z\oplus \Z/5\Z.
\end{equation*}
v. A matrix $W\in\GL_4(\Z)$ such that $\HNF\left(({^3U})^T\right)=W\cdot ({^3U})^T$ is given by
\begin{equation*}
    W=\left(
          \begin{array}{cccc}
            1 & 0 & 0 & 0 \\
            1 & 0 & 1 & -2 \\
            0 & 1 & -3 & 2 \\
            0 & 0 & 1 & -1 \\
          \end{array}
        \right)
\end{equation*}
giving
\begin{eqnarray*}
      G &:=& {_1\widehat{V}}\cdot\ ({_{1}W})^T =\left(
                                                  \begin{array}{c}
                                                    1  \\
                                                  \end{array}
                                                \right)
       \\
    \end{eqnarray*}
Therefore
$$\Ga= {_{1}W} \mod 5= \left(
          \begin{array}{cccc}
            [0]_5 & [4]_5 & [2]_5 & [1]_5 \\
          \end{array}
        \right)\,.$$

\noindent vi. Applying procedure \cite[\S~1.2.3]{RT-LA&GD} as described in part 2 of Theorem \ref{thm:generazione}, one gets a $4\times 4$ matrix $C_X$ whose rows give a basis of $\Cart(X)$ inside $\Weil(X)\cong\Z^{|\Si(1)|}$. Namely
\begin{equation*}
    C_X=\left(
                                    \begin{array}{cccc}
                                      5 & 0 & 0 & 0 \\
                                      0 & 5 & 0 & 0 \\
                                      -3 & -3 & 1 & 0 \\
                                      -2 & -4 & 0 & 1 \\
                                    \end{array}
                                  \right)
\end{equation*}
meaning that $$\Cart(X)=\mathcal{L}\left( 5D_1,5D_2,-3D_1-3D_2+D_3,-2D_1-4D_2+D_4\right)\,.$$

\noindent On the other hand, by part 3 of \cite[Thm.~2.9]{RT-LA&GD}, a basis of
$\Cart(Y)\subseteq\Weil(Y)$ is given by the rows of
$$C_Y=\mathbf{I}_4\cdot U_Q=U_Q\in\GL_n(\Z)$$
giving $\Cart(Y)=\Weil(Y)$, as expected for $Y=\P^3$.

\noindent vii. A basis of $\Pic(X)$ inside $\Cl(X)$ is then  obtained by applying part 6 of Theorem \ref{thm:generazione}. With the notation of Remark \ref{rem:} vii, a switching matrix $A$ such that $A\cdot C_X\cdot Q^T$ is in $\HNF$ is
$$A=\left(
                                    \begin{array}{cccc}
                                      1 & 0 & 0 & 0 \\
                                      -1 & 1 & 0 & 0 \\
                                      1 & 0 & 1 & 0 \\
                                      1 & 0 & 0 & 1 \\
                                    \end{array}
                                  \right)$$

\noindent so that

$$
B_X=\ ^1(A\cdot C_X\cdot Q^T) =\left(
      \begin{array}{c}
        5 \\
      \end{array}
    \right)
$$
$$
\Theta_X=\ ^1(A\cdot C_X\cdot \Gamma^T) =\left(
      \begin{array}{c}
        0 \\
      \end{array}
    \right)
$$
Then
$$\Pic(X)\cong \Z[5d_X(D_1)]\leq\Z[d_X(D_1)]\oplus\Z/5\Z[d_X(D_3-D_4)]\cong\Cl(X)\ \Rightarrow\ \Cl(X)/\Pic(X)\cong\Z/5\Z\oplus\Z/5\Z\,.
$$

\noindent Since $|\SF(V)|=1$, matrices $Q$ and $\Ga$ suffice to give the geometric description of $X$ as Cox geometric quotient. Namely, saying $Z_{\Si}$ the exceptional subset of $\C^{n+r}$ determined by $\Si$, i.e. $Z_{\Si}=\{\mathbf{0}\}\subseteq\C^4$, the columns of the weight matrix $Q$ describe the exponents of the action of $\Hom(F,\C^*)\cong\C^*$ on $\C^4\setminus\{\mathbf{0}\}$ clearly giving $\P^3$, while the columns of the torsion matrix $\Ga$ describe the exponents of the action of $\Hom(\Tors(\Cl(X)),\C^*)\cong \mu_5$ on $Y=\P^3$ whose quotient gives $X$:
\begin{equation}\label{azione}
    \begin{array}{ccc}
       \mu_5\times\P^3 & \longrightarrow & \P^3 \\
       (\varepsilon,[x_1:\ldots :x_4]) & \,\mapsto & \left[x_1:\varepsilon^4 x_2:\varepsilon^2 x_3:\varepsilon x_4\right] \ .
     \end{array}
\end{equation}
As M.~Kasprzyk observes, this example has been firstly presented by M.~Reid as the quotient of $\P^3$ by the action
\begin{equation}\label{azione-Reid}
    \begin{array}{ccc}
       \mu_5\times\P^3 & \longrightarrow & \P^3 \\
       (\varepsilon,[x_1:\ldots :x_4]) & \,\mapsto & \left[\varepsilon x_1:\varepsilon^{2}x_2:\varepsilon^{3} x_3:\varepsilon^{4}x_4\right]
     \end{array}
\end{equation}
(see \cite{Reid85} Ex.~(4.15)).
Recalling Remark \ref{rem:isoquozienti}, to show that actions (\ref{azione}) and (\ref{azione-Reid}) give isomorphic quotients, one can go back obtaining the fan matrix of the geometric quotient (\ref{azione-Reid}) and check that, up to a permutation on columns, it is equivalent to the fan matrix $V$. Then let us follow steps (1), (2) and (3) of Theorem \ref{thm:da-quot-a-fan}. The data of the Reid's quotient (\ref{azione-Reid}) are the weight matrix $Q=\left(
                                               \begin{array}{cccc}
                                                 1 & 1 & 1 & 1 \\
                                               \end{array}
                                             \right)
$ and the torsion matrix $\Ga_R:=\left(
                                   \begin{array}{cccc}
                                     [1]_5 & [2]_5 & [3]_5 & [4]_5 \\
                                   \end{array}
                                 \right)
$. Let $C=\left(
            \begin{array}{cccc}
              1 & 2 & 3 & 4 \\
            \end{array}
          \right)
$ be a matrix of representative integers for $\Ga_R$, as in step (1) of Theorem \ref{thm:da-quot-a-fan}. The the matrix $K$ defined in step (2) of the same Theorem is given by
\begin{equation*}
    K=\left(
        \begin{array}{c}
          \widehat{V}\cdot C^T \\
          5 \\
        \end{array}
      \right) = \left(
                  \begin{array}{c}
                    -4 \\
                    1 \\
                    -1 \\
                    5 \\
                  \end{array}
                \right)
\end{equation*}
Step (3) then gives
\begin{equation*}
    \HNF(K)=\left(
              \begin{array}{c}
                1 \\
                0 \\
                0 \\
                0 \\
              \end{array}
            \right)=U\cdot K\quad\text{with}\quad U=\left(
                      \begin{array}{cccc}
                        0 & 1 & 0 & 0 \\
                        1 & 4 & 0 & 0 \\
                        0 & 1 & 1 & 0 \\
                        2 & 3 & 0 & 1 \\
                      \end{array}
                    \right)
\end{equation*}
Therefore
\begin{equation*}
    \b_R=\,({_3U})_{\{1,2,3\}}=\left(
                               \begin{array}{ccc}
                                 1 & 4 & 0 \\
                                 0 & 1 & 1  \\
                                 2 & 3 & 0 \\
                               \end{array}
                             \right)
\end{equation*}
giving the fan matrix
\begin{equation*}
    V_R:=\b_R\cdot\widehat{V}=\left(
                                \begin{array}{cccc}
                                  1 & 4 & -11 & 6 \\
                                  0 & 1 & -2 & 1 \\
                                  2 & 3 & -7 & 2 \\
                                \end{array}
                              \right)
\end{equation*}
The latter $F$--matrix turns out to be equivalent to the fan matrix $V$ in (\ref{KPRZ-fan}), up to a permutation on columns, since $V_R=R\cdot V\cdot S$ with $R\in\GL_3(\Z)$ and $S\in\GL_4(\Z)$ a permutation matrix, given by
$$
R=\left(
    \begin{array}{ccc}
      1 & -11 & -6 \\
      0 & -2 & -1 \\
      2 & -7 & -4 \\
    \end{array}
  \right)\quad,\quad S=\left(
                         \begin{array}{cccc}
                           1 & 0 & 0 & 0 \\
                           0 & 0 & 1 & 0 \\
                           0 & 1 & 0 & 0 \\
                           0 & 0 & 0 & 1  \\
                         \end{array}
                       \right)\ .
$$}
\end{example}

\begin{example}\label{ex:} {\rm This is a less trivial example for which the procedure of Remark \ref{rem:} turns out to be quite useful to get all the necessary information.

Let us consider the following matrix
\begin{equation}\label{V}
    V=\left(
      \begin{array}{cccccc}
        18 & -21 & -9 & 333 & -492 & 120\\
        -3 & 8 & 4 & -14 & 13 & -4\\
        -23 & 33 & 14 & -404 & 588 & -144\\
       -20 & 26 & 12 & -337 & 493 & -121
            \end{array}
    \right)\ .
\end{equation}
i. A matrix $U_V\in\GL_6(\Z)$ such that $\HNF(V^T)=U_V\cdot V^T$ is given by
\begin{equation*}
    U_V=\left(
          \begin{array}{cccccc}
            20 & 18 & 11 & 24 & 16 & 0\\
             5 & 4 & 4 & 6 & 4 & 0\\
            20 & 19 &  9 & 24 & 16 & 0\\
             12 & 10 &  9 & 15 &  10 &  0\\
            2 & 4 &  1 &  5 &  4 &  3\\
             1 &  1 & 3 & 2 & 3 & 7\\
          \end{array}
        \right)
\end{equation*}
whose last two rows give the $W$--matrix
$$Q=\left(
                                            \begin{array}{cccccc}
                                              2 & 4 &  1 &  5 &  4 &  3\\
                                              1 &  1 & 3 & 2 & 3 & 7\\
                                            \end{array}
                                          \right)\,.
$$
Therefore $V$ is a $F$--matrix \cite[Prop.~3.12(2)]{RT-LA&GD}. Consider $X(\Si)$ with $\Si\in\SF(V)$. Then $X$ is a 4--dimensional $\Q$--factorial complete toric variety of rank 2. One can check that $|\SF(V)|=3$, meaning that we have 3 choices for the fan $\Si$ i.e. for the toric variety $X(\Si)$.

\noindent ii. A switching matrix $U_Q\in\GL_6(\Z)$ such that $\HNF(Q^T)=U_Q\cdot Q^T$ is given by
\begin{equation*}
    U_Q=\left(
          \begin{array}{cccccc}
            5 &  -2 & -1 & 0 & 0 & 0\\
            2 & -1 & 0 & 0 & 0 & 0 \\
            11 & -5 & -2 & 0 & 0 & 0\\
            4 & -3 & -1 & 1 & 0 & 0\\
            7 & -4 & -2 & 0 & 1 & 0\\
            15 & -7 & -5 & 0 & 0 & 1\\
          \end{array}
        \right)
\end{equation*}
encoding a basis $\{d_X(L_1),d_X(L_2)\}$ of a free part $F\cong\Z^2$ of $\Cl(X)$, given by the upper two rows and meaning that
$$L_1=5D_1-2D_2 -D_3\quad,\quad L_2=2D_1 - D_2 \,,$$
and a fan matrix $\widehat{V}$ of the universal 1--covering $Y$, given by the further rows, that is,
\begin{equation}\label{VV}
   \widehat{V}=\left(
          \begin{array}{cccccc}
            11 & -5 & -2 & 0 & 0 & 0\\
            4 & -3 & -1 & 1 & 0 & 0\\
            7 & -4 & -2 & 0 & 1 & 0\\
            15 & -7 & -5 & 0 & 0 & 1\\
          \end{array}
        \right)\,.
\end{equation}
iii. To get a matrix $\b$ such that $\b\cdot \widehat{V}=V$ we need the $\HNF$ of both $\widehat{V}$ and $V$:
\begin{eqnarray*}
    \widehat{H}&:=&\HNF\left(\widehat{V}\right)=\left(
          \begin{array}{cccccc}
            1 & 0 & 0 & 16 & -25 & 6\\
            0 & 1 & 0 & 7 & -12 & 3\\
            0 & 0 & 1 & 4 & -6 & 1\\
            0 & 0 & 0 & 19 & -29 & 7\\
          \end{array}
        \right)\\
    \widehat{U}&=& \left(
          \begin{array}{cccc}
           2 & 16 & -25 & 6\\
            1 & 7 & -12 & 3\\
            1 & 4 & -6 & 1\\
            2 & 19 & -29 & 7\\
          \end{array}
        \right)\in\GL_4(\Z): \quad  \widehat{U}\cdot \widehat{V} = \widehat{H}\,.
\end{eqnarray*}
\begin{eqnarray*}
    H&:=&\HNF(V)=\left(
          \begin{array}{cccccc}
            1 & 0 & 2 & 100 & -153 & 36\\
            0 & 1 & 2 & 53 & -82 & 19\\
            0 & 0 & 3 & 69 & -105 & 24\\
            0 & 0 & 0 & 285 & -435 & 105\\
          \end{array}
        \right)\\
    U&=& \left(
          \begin{array}{cccc}
            -2 & 3 & -2 & 0\\
            0 & 1 & -1 & 1\\
            3 & -1 & -1 & 4\\
            -40 & 34 & -14 & -25\\
          \end{array}
        \right)\in\GL_4(\Z): \quad  U\cdot V = H\,.
\end{eqnarray*}
The given recursive relations (\ref{ricorsive}) define
\begin{equation*}
    \b_H=\left(
          \begin{array}{cccc}
            1 & 0 & 2 & 4\\
            0 & 1 & 2 & 2\\
            0 & 0 & 3 & 3\\
            0 & 0 & 0 & 15\\
          \end{array}
        \right)\ \Rightarrow\ \b=U^{-1}\cdot\b_H\cdot \widehat{U}=\left(
          \begin{array}{cccc}
            30 & 333 & -492 & 120\\
            2 & -14 & 13 & -4\\
            -33 & -404 & 588 & -144\\
           -28 & -337 & 493 & -121\\
          \end{array}
        \right)\,.
\end{equation*}
iv. Therefore $\Delta=\SNF(\b)$ and $\mu,\nu\in\GL_4(\Z):\Delta=\mu\cdot\b\cdot\nu$ are given by
\begin{eqnarray*}
    \Delta&=&\diag(1,1,3,15)\\
    \mu&=&\left(
          \begin{array}{cccc}
            1410 & -1138 & 551 & 780\\
            1140 & -916 & 420 & 661\\
            -1623 & 1304 & -598 & -941\\
            8425 & -6769 & 3104 & 4885\\
          \end{array}
        \right)\\
    \nu&=&\left(
          \begin{array}{cccc}
            1&58&2224&2022\\
            0&1&27&24\\
            0&0&1&1\\
            0&-2&-78&-71\\
          \end{array}
        \right)
\end{eqnarray*}
Then $s=2$ and the last two rows of
\begin{equation*}
    \widehat{V}'=\nu^{-1}\widehat{V}=\left(
          \begin{array}{cccccc}
            521&-251&-168&-2&14&28\\
            388&-222&-112&7&45&3\\
            -184&105&53&-2&-23&-1\\
            191&-109&-55&2&24&1\\
          \end{array}
        \right)
\end{equation*}
give the generators of
$$\Tors(\Cl(X))\cong \Z/3\Z\oplus\Z/15\Z\,,$$
namely $d_X(T_1)$ and $d_X(T_2)$ with
\begin{eqnarray*}
  T_1 &=& -184 D_1+ 105 D_2 +53 D_3-2 D_4-23 D_5-D_6 \\
  T_2 &=& 191 D_1-109 D_2-55 D_3+2 D_4+ 24 D_5 + D_6
\end{eqnarray*}
Hence $$\Cl(X)=\Z\cdot f_X(L_1)\oplus \Z\cdot f_X(L_2)\oplus\Z\cdot \tau_X(T_1)\oplus \Z\cdot \tau_X(T_2)\cong \Z^2\oplus \Z/3\Z\oplus\Z/15\Z\,.$$

\noindent v. A matrix $U$ as defined in part 6 of Theorem \ref{thm:generazione} is given by
\begin{equation*}
    U=\left(
        \begin{array}{c}
          ^2U_Q \\
          \widehat{V}' \\
        \end{array}
      \right)
    =\left(
          \begin{array}{cccccc}
          2&-1&0&0&0&0\\
          -6&3&1&0&0&0\\
            521&-251&-168&-2&14&28\\
            388&-222&-112&7&45&3\\
            -184&105&53&-2&-23&-1\\
            191&-109&-55&2&24&1\\
          \end{array}
        \right)
\end{equation*}
A matrix $W\in\GL_6(\Z)$ such that $\HNF(({^4U})^T)=W\cdot(({^4U})^T)$ is given by
\begin{equation*}
    W=\left(
          \begin{array}{cccccc}
            -57&-115&3&-549&17&0\\
            4&8&1&3&7&0\\
            -125&-250&0&-1090&14&0\\
            -170&-340&0&-1482&19&0\\
            -188&-376&0&-1639&21&0\\
            -126&-252&0&-1092&13&1\\
          \end{array}
        \right)
\end{equation*}
then
\begin{equation*}
    G={_2\widehat{V}'}\cdot ({_2W})^T = \left(
                                          \begin{array}{cc}
                                            -2093&-1392\\
                                            2302&1531\\
                                          \end{array}
                                        \right)
\end{equation*}
A matrix $U_G\in\GL_2(\Z)$ such that $\HNF(G^T)=U_G\cdot G^T$ is given by
\begin{equation*}
    U_G=\left(
          \begin{array}{cc}
            1531&-2302\\
            1392&-2093\\
          \end{array}
        \right)
\end{equation*}
hence giving
\begin{eqnarray*}
  \Ga &=& {U_G}\ \cdot\ {_2W} \mod \boldsymbol\tau \\
  &=& \left(
                                                      \begin{array}{cccccc}
                                                        2224&4448&0&4475&2225&-2302\\ 2022&4044&0&4068&2023&-2093
                                                      \end{array}
                                                    \right)\mod \left(
                                                                  \begin{array}{c}
                                                                    3 \\
                                                                    15 \\
                                                                  \end{array}
                                                                \right)\\
   &=&
   \left(
          \begin{array}{cccccc}
          [1]_3&[2]_3   &[0]_3   &[2]_3   &[2]_3    &[2]_3\\
          {[12]_{15}} &[9]_{15}&[0]_{15}&[3]_{15}&[13]_{15}&[7]_{15}\\
          \end{array}
        \right)
\end{eqnarray*}
vi. Depending on the choice of the fan $\Si_i\in\SF(V)$, by applying procedure \cite[\S~1.2.3]{RT-LA&GD} as described in part 2 of Theorem \ref{thm:generazione}, one gets a $6\times 6$ matrix $C_{X,i}$ whose rows give a basis of $\Cart(X_i)$ inside $\Weil(X_i)\cong\Z^{|\Si_i(1)|}$. Namely
\begin{equation*}
    C_{X,1}=\left(\begin {array}{cccccc}
    265926375&0&0&0&0&0\\
    -148978500&825&0&0&0&0\\
    -58474020&-375&15&0&0&0\\
    37&-18&-7&1&0&0\\
    -58473933&-417&-3&0&3&0\\
    19&-8&-5&0&-1&1\end {array}
                                                                                                 \right)
\end{equation*}
\begin{equation*}
    C_{X,2}=\left( \begin {array}{cccccc}
    43543500&0&0&0&0&0\\
    -34716000&15&0&0&0&0\\
    -594165&0&30&0&0&0\\
    -34715963&-3&-7&1&0&0\\
    17655087&-12&-18&0&3&0\\
    19&-8&-5&0&-1&1
    \end {array} \right)
\end{equation*}
\begin{equation*}
    C_{X,3}=\left(\begin {array}{cccccc}
    43543500&0&0&0&0&0\\
    -37009500&825&0&0&0&0\\
    -6534165&-750&30&0&0&0\\
    37&-18&-7&1&0&0\\
    87&-42&-18&0&3&0\\
    19&-8&-5&0&-1&1\end {array}
                                                                                                 \right)
\end{equation*}
On the other hand, by parts (2) and (3) of \cite[Thm.~2.9]{RT-LA&GD}, one gets matrices $B_{Y,i}$, whose rows give a basis of $\Pic(Y_i)\leq\Cl(Y_i)$, and consequently matrices
\begin{equation*}
   \forall\,1\leq i\leq 3\quad C_{Y,i}=\begin{pmatrix}B_i & \mathbf{0}_{2,4}\\  \mathbf{0}_{4,2}& \mathbf{I}_4\end{pmatrix}\cdot U_Q
\end{equation*}
whose rows give a basis of $\Cart(Y_i)\leq\Weil(Y_i)$. Namely:
\begin{equation*}
    B_{Y,1}=\left(\begin{array}{cc}
                     5909475&0\\
                     -238040&165
                            \end{array}
                            \right)\,,\, B_{Y,2}=\left(\begin{array}{cc}
                            5805800&0\\
                            -4648596&1
                            \end{array}
                            \right)\,,\,B_{Y,3}=\left(\begin{array}{cc}
                            5805800&0\\
                            -217580&55\\
                            \end{array}
                            \right)
\end{equation*}
\begin{equation*}
    C_{Y,1}=\left(
                                                                                                   \begin{array}{cccccc}
                                                                                                     29547375&-11818950&-5909475&0&0&0\\
                                                                                                     -1189870&475915&238040&0&0&0\\
                                                                                                     11&-5&-2&0&0&0\\
                                                                                                     4&-3&-1&1&0&0\\
                                                                                                     7&-4&-2&0&1&0\\
                                                                                                     15&-7&-5&0&0&1\\
                                                                                                   \end{array}
                                                                                                 \right)
\end{equation*}
\begin{equation*}
    C_{Y,2}=\left(
                                                                                                   \begin{array}{cccccc}
                                                                                                     29029000&-11611600&-5805800&0&0&0\\
                                                                                                     -23242978&9297191&4648596&0&0&0\\
                                                                                                     11&-5&-2&0&0&0\\
                                                                                                     4&-3&-1&1&0&0\\
                                                                                                     7&-4&-2&0&1&0\\
                                                                                                     15&-7&-5&0&0&1\\
                                                                                                   \end{array}
                                                                                                 \right)
\end{equation*}
\begin{equation*}
    C_{Y,3}=\left(
                                                                                                   \begin{array}{cccccc}
                                                                                                     29029000&-11611600&-5805800&0&0&0\\
                                                                                                     -1087790&435105&217580&0&0&0\\
                                                                                                     11&-5&-2&0&0&0\\
                                                                                                     4&-3&-1&1&0&0\\
                                                                                                     7&-4&-2&0&1&0\\
                                                                                                     15&-7&-5&0&0&1\\
                                                                                                   \end{array}
                                                                                                 \right)
\end{equation*}

\noindent vii.  A basis of $\Pic(X_i)$ inside $\Cl(X_i)$ is then  obtained by applying part 6 of Theorem \ref{thm:generazione}. For $i=1,2,3$, matrices $A_i$ switching  $C_{X_i}\cdot Q^T$ in Hermite normal form are respectively
\begin{equation*}
A_1=\left(
      \begin{array}{cccccc}
       -351039&-449987&-449987&0&0&0\\
       -502913&-644670&-644670&0&0&0\\
       1&1&2&0&0&0\\
       0&0&0&1&0&0\\
       1&1&1&0&1&0\\
       0&0&0&0&0&1
      \end{array}
    \right)\end{equation*}

\begin{equation*}
A_2=\left(
      \begin{array}{cccccc}
       -93838&-117699&0&0&0&0\\
       -1157199&-1451450&0&0&0&0\\
       4&5&1&0&0&0\\
       0&-1&0&1&0&0\\
       -2&-2&0&0&1&0\\
       0&0&0&0&0&1
      \end{array}
    \right)\end{equation*}

    \begin{equation*}
    A_3=\left(
          \begin{array}{cccccc}
          -10317&-12139&0&0&0&0\\
          -22429&-26390&0&0&0&0\\
          1&1&1&0&0&0\\0&0&0&1&0&0\\
          0&0&0&0&1&0\\
          0&0&0&0&0&1
          \end{array}
        \right)\end{equation*}

giving
\begin{eqnarray*}
  B_{X_1} &=& \ ^2(A_1\cdot C_{X_1}\cdot Q^T) =\left(
      \begin{array}{cc}
       825&185620050\\
      0&265926375
      \end{array}
    \right) \\
  B_{X_2} &=&  \ ^2(A_2\cdot C_{X_2}\cdot Q^T) =\left(
      \begin{array}{cc}
       60&1765515\\
       0&21771750
      \end{array}
    \right) \\
  B_{X_3} &=&  \ ^2(A_3\cdot C_{X_3}\cdot Q^T) =\left(
      \begin{array}{cc}
       3300&10016325\\
       0&21771750
      \end{array}
    \right)\\
  \Theta_{X_i}&=&   \ ^2(A_i\cdot C_{X_i}\cdot \Ga^T) =\left(
                                                         \begin{array}{cc}
                                                           \,[0]_3 & [0]_{15} \\
                                                           \,[0]_3 & [0]_{15} \\
                                                         \end{array}
                                                       \right)\ ,\quad\text{for $i=1,2,3$\,.}
\end{eqnarray*}

\noindent Finally, matrices $Q$ and $\Ga$  give the geometric description of $X(\Si)$ as a Cox geometric quotient. Namely, saying $Z=Z_{\Si}=Z_{\widehat{\Si}}$ the exceptional subset of $\C^{n+r}$ determined by $\Si\in\SF(V)$ (or equivalently by $\widehat{\Si}\in\SF(\widehat{V})$), whose explicit determination is left to the reader, the columns of the weight matrix $Q$ describe the exponents of the action of $\Hom(F,\C^*)\cong(\C^*)^2$ on $\C^6\setminus Z$, whose quotient gives $Y$; on the other hand the columns of the torsion matrix $\Ga$ describe the exponents of the action of $\Hom(\Tors(\Cl(X)),\C^*)\cong \mu_{3}\oplus\mu_{15}$ on $Y$, whose quotient gives $X$. Putting all together one gets the following action of $G(X)=\Hom(\Cl(X),\C^*)\cong(\C^*)^2\oplus\mu_{3}\oplus\mu_{15}$
\begin{equation*}
    g: \left((\C^*)^2\oplus\mu_{3}\oplus\mu_{15}\right)\times \left(\C^6\setminus Z\right) \longrightarrow  \C^6\setminus Z
\end{equation*}
defined by setting
\begin{eqnarray}\label{azione_g}
    &&g\left(((t_1,t_2),\varepsilon,\eta),(x_1,\ldots :x_6)\right):=\\
    \nonumber
    &&\left(t_1^2t_2\varepsilon\eta^{12}\ x_1,t_1^4t_2\varepsilon^2\eta^9\  x_2,t_1t_2^3 \ x_3, t_1^5t_2^2\varepsilon^2\eta^3\ x_4,t_1^4t_2^3\varepsilon^2\eta^{13}\ x_5,t_1^3t_2^7\varepsilon^2\eta^7\ x_6\right)
\end{eqnarray}}
\end{example}

\begin{example} {\rm This example is devoted to give an application of Theorem \ref{thm:da-quot-a-fan} by reversing the previous Example \ref{ex:}.

Let $X$ be the $\Q$--factorial complete toric variety defined by the toric geometric quotient
\begin{equation*}
    X=\left(\C^6\setminus Z\right)\,\left/_g\,\left((\C^*)^2\oplus\mu_{3}\oplus\mu_{15}\right)\right.
\end{equation*}
where $g$ is the action defined in (\ref{azione_g}). Our aim is recovering a fan matrix of $X$.

From the exponents of the action $g$ we get the two matrices:
\begin{eqnarray*}
  \text{the $W$--matrix} && Q=\left(
                                            \begin{array}{cccccc}
                                              2 & 4 &  1 &  5 &  4 &  3\\
                                              1 &  1 & 3 & 2 & 3 & 7\\
                                            \end{array}
                                          \right) \\
  \text{the torsion matrix} && \Ga=\left(
          \begin{array}{cccccc}
          [1]_3&[1]_3   &[1]_3   &[0]_3   &[0]_3    &[0]_3\\
          {[8]_{15}} &[8]_{15}&[3]_{15}&[4]_{15}&[13]_{15}&[0]_{15}\\
          \end{array}
        \right)
\end{eqnarray*}
By Gale duality one gets immediately the $CF$--matrix $\widehat{V}=\G(Q)$ given in (\ref{VV}), which is a fan matrix of the PWS $Y$ giving the universal 1--covering of $X$. Hence $V$ is obtained by recovering a matrix $\b$ such that $V=\b\cdot \widehat{V}$, as in (\ref{universalità}). At this purpose let us apply Theorem \ref{thm:da-quot-a-fan}. Let
\begin{equation*}
    C=\left(
          \begin{array}{cccccc}
          1&1   &1   &0   &0    &0\\
          8 &8&3&4&13&0\\
          \end{array}
        \right)
\end{equation*}
be a matrix of representative integers of $\Ga$ and construct the matrix
\begin{equation*}
    K:=\left(
                                  \begin{array}{c}
                                    \widehat{V}\cdot C^T \\
                                    3\,\quad\, 0\\
                                    0\quad 15\\
                                  \end{array}
                                \right)=\left(
                                          \begin{array}{cc}
                                            4 & 42 \\
                                            0 & 9 \\
                                            1 & 31 \\
                                            3 & 49 \\
                                            3 & 0 \\
                                            0 & 15 \\
                                          \end{array}
                                        \right)
\end{equation*}
Then
\begin{equation*}
    \HNF(K)=\left(
              \begin{array}{cc}
                1 & 0 \\
                0 & 1 \\
                0 & 0 \\
                0 & 0 \\
                0 & 0 \\
                0 & 0 \\
              \end{array}
            \right)= U\cdot K\quad \text{with}\quad U=\left(
                                                \begin{array}{cccccc}
                                                  -5&-49&21&0&0&0\\
                                                  -1&-9&4&0&0&0\\
                                                  -9&-82&36&0&0&0\\
                                                  -8&-68&29&1&0&0\\
                                                  -3&-17&9&0&1&0\\
                                                  -3&-29&12&0&0&1\\
                                                \end{array}
                                              \right)
\end{equation*}
giving
\begin{equation*}
    \b=({_4U})_{\{1,2,3,4\}}=\left(
                               \begin{array}{cccc}
                                 -9&-82&36&0\\
                                                  -8&-68&29&1\\
                                                  -3&-17&9&0\\
                                                  -3&-29&12&0 \\
                               \end{array}
                             \right)
\end{equation*}
Therefore a fan matrix of $X$ is given by
\begin{equation*}
    V'':=\b\cdot\widehat{V}= \left(
                          \begin{array}{cccccc}
                            -175&147&28&-82&36&0\\
                            -142&121&21&-68&29&1\\
                            -38&30&5&-17&9&0\\
                            -65&54&11&-29&12&0\\
                          \end{array}
                        \right)
\end{equation*}
As a final remark, let us notice that the so obtained matrix $V''$ is actually equivalent to the matrix $V$ in (\ref{V}): in fact $R\cdot V=V''$ for
\begin{equation*}
    R=\left(
        \begin{array}{cccc}
          -646&512&-203&-416\\
          -533&422&-166&-345\\
          -143&113&-45&-92\\
          -237&188&-75&-152\\
        \end{array}
      \right)\in \GL_4(\Z)\,.
\end{equation*}}

\end{example}

\begin{acknowledgements}
We would like to thank Antonella Grassi for helpful conversations and suggestions. We also thank the unknown Referee whose comments sensibly improved the paper: in particular he suggested us the above proof of Thm.~\ref{thm:covering&quotient} which dramatically shortened our original proof.
\end{acknowledgements}



\newpage
\thispagestyle{empty}
\topskip0pt
\vspace*{\fill}
 In the following we report the erratum correcting only those parts of our paper published in \emph{Ann. Mat. Pur. Appl.} \textbf{196} (2017), 325--347, which are affected by the error in Proposition 3.1. It will be published soon in the Annali di Matematica Pura ed Applicata.
\vspace*{\fill}

\newpage
\setcounter{page}{1}

 \vskip1cm
\noindent{\Large \bf{Erratum to: A $\Q$--factorial complete toric variety is a quotient of a poly weighted space}}
\oneline
\centerline{\textbf{Michele Rossi and Lea Terracini} }
\oneline\oneline

\setcounter{section}{3}

After the publication of \cite{RT-QUOT}, we realized that Proposition~3.1, in that paper, contains an error, whose consequences are rather pervasive along the whole section 3 and for some aspects of examples 5.1 and 5.2. Here we give a complete account of needed corrections.

First of all \cite[Prop.~3.1]{RT-QUOT} has to be replaced by the following:

\begin{proposition}\label{prop:CartierYX} Let $X(\Si)$ be a $Q$--factorial complete toric variety and $Y(\widehat{\Si})$ be its universal 1-covering. Let $\{D_{\rho}\}_{\rho\in\Si(1)}$ and $\{\widehat{D}_{\rho}\}_{\rho\in\widehat{\Si}(1)}$ be the standard bases of $\Weil(X)$ and $\Weil(Y)$, respectively, given by the torus orbit closures of the rays. Then
\begin{equation*}
     D=\sum_{\rho\in\Si(1)} a_{\rho} D_{\rho}\in\Cart(X)\quad \Longrightarrow\quad \widehat{D}=\sum_{\rho\in\widehat{\Si}(1)} a_{\rho} \widehat{D}_{\rho}\in\Cart(Y)\,.
\end{equation*}
Therefore, under the identification $\Z^{|\Si(1)|}\cong\Weil(X)\stackrel{\a}{\cong}\Weil(Y)\cong\Z^{|\widehat{\Si}(1)|}$ realized by the isomorphism $D_{\rho}\stackrel{\a}{\mapsto}\widehat{D}_{\rho}$,
$$
\Cart(X)\cong\a(\Cart(X))\leq\Cart(Y)\leq\Weil(Y)
$$
is a chain of subgroup inclusions. Moreover the induced morphism $\overline{\a}:\Cl(X)\to\Cl(Y)$ is injective when restricted to $Pic(X)$, realizing the following further chain of subgroup inclusions
$$\Pic(X)\cong\overline{\a}(\Pic(X))\leq\Pic(Y)\leq\Cl(Y)$$.
\end{proposition}

\begin{proof} Let us fix a basis $\mathcal{B}$ of the $\Z$-module $M\cong\Z^n$ and let $V$ and $\widehat{V}$ be fan matrices representing the standard morphisms
$$
div_X:M\cong\Z^n \stackrel{V^T}\longrightarrow\Z^{|\Si(1)|}\cong\Weil(X)\quad,\quad div_Y:M\cong\Z^r \stackrel{\widehat{V}^T}\longrightarrow\Z^{|\widehat{\Si}(1)|}\cong\Weil(Y)
$$
Let $\b\in\GL_n(\Q)\cap\mathbf{M}_n(\Z)$ be  such that $V=\b\widehat{V}$ and so realizing an injective endomorphism of the $\Z$-module $M$.
The result follows by writing down the condition of being locally principal for a Weil divisor and observing that
\begin{eqnarray}\label{ISigma}
    \mathcal{I}^\Si&=&\{I\subseteq\{1,\ldots,n+r\}:\left\langle V^I\right\rangle\in\Si(n)\}\\
    \nonumber
    &=& \{I\subseteq\{1,\ldots,n+r\}:\left\langle \widehat{V}^I\right\rangle\in\widehat{\Si}(n)\}=\mathcal{I}^{\widehat{\Si}}
\end{eqnarray}
by the construction of $\widehat{\Si}\in\SF(\widehat{V})$, given the choice of $\Si\in\SF(V)$. Notice that $\mathcal{I}^\Si$ describes the complements of those sets described by $\mathcal{I}_\Si$, as defined in \cite[Rem.\,2.4]{RT-QUOT}.  In particular the Weil divisor $\sum_{j=1}^{n+r}a_jD_j\in\Weil(X)$ is Cartier if and only if
\begin{equation}\label{cartier}
    \forall\,I\in\mathcal{I}^\Si\quad\exists\,\mathbf{m}_I\in M : \forall\,j\not\in I\ \v_j^T\mathbf{m}_I=a_j\,,
\end{equation}
where $\v_j$ is the $j$-th column of $V$.
Then $\a(\sum_{j=1}^{n+r}a_jD_j)=\sum_{j=1}^{n+r}a_j\widehat{D}_j$ is a Cartier divisor since
\begin{equation*}
  \forall\,I\in\mathcal{I}^\Si\quad \forall\,j\not\in I \quad \widehat{\v}_j^T(\b^T\mathbf{m}_I)=a_j
\end{equation*}
where $\widehat{\v}_j$ is the $j$-th column of $\widehat{V}$. \\
The injectivity of $\overline{\alpha}$ follows from the well-known freeness of $\Pic(X)$.
\end{proof}

As a consequence, parts 1, 4, 5 of \cite[Thm.~3.2]{RT-QUOT} still hold, while parts 2, 3, 6, 7 have to be replaced by the following:

\begin{theorem}\label{thm:generazione} Let $X=X(\Si)$ be a $n$--dimensional $\Q$--factorial complete toric variety of rank $r$ and $Y=Y(\widehat{\Si})$ be its universal 1--covering. Let $V$ be a reduced fan matrix of $X$, $Q=\G(V)$ a weight matrix of $X$ and $\widehat{V}=\G(Q)$ be a $CF$--matrix giving a fan matrix of $Y$.
\begin{itemize}
\item[2.] Define $\mathcal{I}^\Si$ as in (\ref{ISigma}). For any $I\in\mathcal{I}^\Si$ let $E_I$ be the $r\times (n+r)$ matrix admitting as rows the standard basis vectors $e_i=(0,\ldots,0,\underset{i}{1},0,\ldots,0)$, for $i\in I$, representing the $i$-th basis divisor $D_i\in\Weil(X)\cong\Z^{|\Si(1)|}$. Set $\widetilde{V}_I:=\left(V^T\,|\,E_I^T\right)\in\mathbf{M}_{n+r}(\Z)$. Then Cartier divisors give rise to the following maximal rank subgroup of $\Weil(X)$
    \begin{equation*}
      \Cart(X)\cong \bigcap_{I\in\mathcal{I}^\Si} \mathcal{L}_c\left(\widetilde{V}_I\right)\leq \Z^{|\Si(1)|}\cong\Weil(X)
    \end{equation*}
  and a basis of $\Cart(X)\leq\Weil(X)$ can be explicitly computed by applying the procedure described in \cite[\S~1.2.3]{RT-LA&GD}.
\item[3.] Let $C_X\in\GL_{n+r}(\Q)\cap\mathbf{M}_{n+r}(\Z)$ be a matrix whose rows give a basis of $\Cart(X)$ in $\Weil(X)$, as obtained in the previous part 2. Identify $\Cl(X)$ with $\Z^r\oplus\bigoplus_{k=1}^s\Z/\tau_k\Z$ by item (c) of part 4 in \cite[Thm.~3.2]{RT-QUOT}, and represent the morphism $d_X$ by $Q\oplus \Gamma$, according to parts 1 and 5. Let $A\in\GL_{n+r}(\Z)$ be a matrix such that $A\cdot C_X \cdot Q^T$ is in $\HNF$. Let $\mathbf{c}_1,\ldots,\mathbf{c}_r$ be the first $r$ rows of the matrix $A\cdot C_X$ and for $i=1,\ldots r$ put $\mathbf{b}_i=Q\cdot\mathbf{c}_i^T + \Ga\cdot \mathbf{c}_i^T$. Then $\mathbf{b}_1,\ldots \mathbf{b}_r$ is a basis of the free group $\Pic(X)$ in $\Cl(X)$.
\item[6.] Given the choice of $\widehat{V}$ and $V$ as in the previous parts 4 and 5 of \cite[Thm.~3.2]{RT-QUOT}, consider
    \begin{eqnarray*}
      U&:=&\left(
           \begin{array}{c}
             ^rU_Q \\
             \widehat{V} \\
           \end{array}
         \right)\in\GL_{n+r}(\Z)\\
      W &\in&\GL_{n+r}(\Z) \ :\  W\cdot ({^{n+r-s}U})^T=\HNF\left(({^{n+r-s}U})^T\right) \\
      G &:=& {_s\widehat{V}}\cdot\ ({_{s}W})^T \in \mathbf{M}_s(\Z)\\
      U_G&\in&\GL_{s}(\Z) \ :\  U_G\cdot G^T =\HNF(G^T)\,.
    \end{eqnarray*}
    Then a ``torsion matrix''  representing the ``torsion part'' of the morphism $d_X$, that is, $\tau_X:\Weil(X)\to\Tors(\Cl(X))$,
    is given by
    \begin{equation}\label{Gamma}
        \Ga = {U_G}\cdot\ {_{s}W} \mod {\boldsymbol\tau}
    \end{equation}
    where this notation means that the $(k,j)$--entry of $\Ga$ is given by the class in $\Z/\tau_k\Z$ represented by the corresponding $(k,j)$--entry of ${^sU_G}\cdot\ {_{s}W}$, for every $1\leq k\leq s\,,\,1\leq j\leq n+r$.
\item[7.] Setting $\d_{\Si}:=\lcm\left(\det(Q_I):I\in\mathcal{I}^\Sigma\right)$
then
$$\d_{\Si}\mathcal{W}_T(X)\subseteq \mathcal{C}_T(X)\quad\text{and}\quad\d_{\Si}\mathcal{W}_T(Y)\subseteq \mathcal{C}_T(Y)$$
and there are the following divisibility relations
$$\d_{\Si}\ |\ [\Cl(Y):\Pic(Y)]=[\mathcal{W}_T(Y):\mathcal{C}_T(Y)]\ |\ [\Cl(X):\Pic(X)]= [\mathcal{W}_T(X):\mathcal{C}_T(X)]\,.$$
\end{itemize}
\end{theorem}

\begin{proof}
(2): Recalling relation (\ref{cartier}) in the proof of Proposition \ref{prop:CartierYX}, set
$$\forall\,I\in\mathcal{I}^\Si\quad\mathcal{P}^I=\{L=\sum_{j=1}^{n+r}a_jD_j\in \mathcal{W}_T(X)\ |\ \exists\,\mathbf{m}\in M : \forall\,j\not\in I\ \mathbf{m}\cdot\v_j=a_j\}.$$
Then $\mathcal{P}^I$ contains $\mathrm{Im}(div_X:M\to\Weil(X))=\mathcal{L}_c\left(V^T\right)$ and a $\Z$-basis of $\mathcal{P}^I$ is given by
$$\{D_j, j\in I\}\cup\{\sum_{k=1}^{n+r}v_{ik}D_k, i=1,\ldots ,n\},$$
where $\{v_{ik}\}$ is the $i$-th entry of $\mathbf{v}_k$, so giving the rows of the matrix $\widetilde{V}_I$ defined in the statement.

\noindent (3): By definition $$\Pic(X)=\mathrm{Im}(\Cart(X)\hookrightarrow\Weil(X)\stackrel{d_X}{\to}\Cl(X))$$
so that $\Pic(X)$ is generated by the image under $Q\oplus \Gamma$ of the transposed of the rows of $C_X$. Since $\rk(C_X)=n+r$ and $\rk(Q)=r$, the matrix $C_X\cdot Q^T$ has rank $r$ and therefore its $\HNF$ has the last $n-r$ rows equal to zero. Therefore the rows of the matrix $A\cdot C_X$ provide a basis of $\Cart(X)$ in $\Weil(X)$ such that its last $n$ rows are a basis of $\mathcal{L}_r(\widehat V)\cap \Cart(X)=\mathcal{L}_r( V)$. Since $\Pic(X)$ is free of rank $r$ it is freely generated by the images under $d_X$ of the first $r$ rows.

\noindent (6): A representative matrix of the torsion part $\tau_X:\Weil(X)\to\Cl(X)$ of the morphism $d_X$ is any matrix satisfying the following properties:
 \begin{itemize}
   \item[$(i)$] $\Ga=(\g_{kj})$ with $\g_{kj}\in\Z/\tau_k\Z$,
   \item[$(ii)$] $\Ga\cdot (^rU_Q)^T=\mathbf{0}_{s,r} \mod \boldsymbol\tau$, meaning that $\Ga$ kills the generators of the free part $F\leq\Cl(X)$ defined in display (4) of part 1 of \cite[Thm.~3.2]{RT-QUOT},
   \item[$(iii)$] $\Ga\cdot V^T=\mathbf{0}_{s,n} \mod \boldsymbol\tau$, where $V$ is a fan matrix satisfying condition   4.(b) in \cite[Thm.~3.2]{RT-QUOT}: this is due to the fact that the rows of $V$ span $\ker(d_X)$,
   \item[$(iv)$] $\Ga\cdot({_s\widehat{V}})^T=\mathbf{I}_s \mod \boldsymbol\tau$, since the rows of ${_s\widehat{V}}$ give the generators of $\Tors(\Cl(X))$, as in display (6) of part 5 of \cite[Thm.~3.2]{RT-QUOT}.
 \end{itemize}
 Therefore it suffices to show that the matrix $ {U_G}\cdot\ {_{s}W}$ in (\ref{Gamma}) satisfies the previous conditions $(ii)$, $(iii)$ and $(iv)$ without any reduction mod $\boldsymbol\tau$, that is,
 \begin{equation*}
    {U_G}\cdot\ {_{s}W}\cdot\ ({^{n+r-s}U})^T=\mathbf{0}_{s,n+r-s}\quad,\quad {U_G}\cdot\ {_{s}W}\cdot\ ({_s\widehat{V}})^T=\mathbf{I}_s\,.
 \end{equation*}
The first equation follows by the definition of $W$, in fact
\begin{equation*}
    W\cdot ({^{n+r-s}U})^T=\HNF\left(({^{n+r-s}U})^T\right)
    =\left(
                                                                                             \begin{array}{c}
                                                                                               \mathbf{I}_{n+r-s} \\
                                                                                               \mathbf{0}_{s,n+r-s} \\
                                                                                             \end{array}
                                                                                           \right)\,\Rightarrow\,{_{s}W}\cdot\ ({^{n+r-s}U})^T=\mathbf{0}_{s,n+r-s}
\end{equation*}
The second equation follows by the definition of $U_G$, in fact
\begin{equation*}
    U_G\cdot {_{s}W}\cdot\ ({_s\widehat{V}})^T= U_G\cdot G^T =\HNF(G^T)= \mathbf{I}_s\,.
\end{equation*}

\noindent (7): Part (4) of \cite[Thm.~2.9]{RT-LA&GD} gives that $\d_{\Si}\ |\  [\Cl(Y):\Pic(Y)]=[\Weil(Y):\Cart(Y)]$. On the other hand Proposition \ref{prop:CartierYX} gives that $[\mathcal{W}_T(Y):\mathcal{C}_T(Y)]\ |\ [\mathcal{W}_T(X):\mathcal{C}_T(X)]=[\Cl(X):\Pic(X)]$.
\end{proof}

\halfline
Considerations i, ii, iii, iv, v of \cite[Rem.~3.3]{RT-QUOT} still holds, while vi, vii and the remaining part of Remark 3.3 have to be replaced by the following

\begin{remark}\label{rem:} {\rm  \hfill
\begin{itemize}

\item[vi.] apply procedure \cite[\S~1.2.3]{RT-LA&GD}, based on the $\HNF$ algorithm, to get a $(n+r)\times (n+r)$ matrix $C_X$ whose rows give a basis of $\Cart(X)\leq\Weil(X)\cong\Z^{|\Si(1)|}$;
  \item[vii.] apply procedure described in part 6 of Theorem \ref{thm:generazione} to get a system of  generators of $\Pic(X)$ in $ \Cl(X)\,.$ Precisely, let $A\in\GL_{n+r}(\Z)$ be a switching matrix such that $\HNF(C_X\cdot Q^T)=A\cdot C_X \cdot Q^T$, and put
  \begin{equation}
  \label{eq:BX}
  B_X=\ ^r(A\cdot C_X \cdot Q^T),\quad \Theta_X=\ ^r{(A\cdot C_X \cdot \Gamma^T)}\end{equation} \\
 Then the rows of the matrices $B_X$ and $\Theta_X$ represent  respectively the free part and the torsion part of a basis of $\Pic(X)$ in $\Cl(X)$, where the latter is identified to $ \Z^r\oplus\bigoplus_{k=1}^s\Z/\tau_k\Z$.

Moreover:
\begin{itemize}
  \item recall that, for the universal 1--covering $Y$ of $X$, once fixed the basis $\{\widehat{D}_j\}_{j=1}^{n+r}$ of $\mathcal{W}_T(Y)\cong\Z^{n+r}$ and the basis $\{d_Y(\widehat{L}_i)\}_{i=1}^r$ of $\Cl(Y)\cong\Z^r$, (see (11) in \cite[Thm.~2.9]{RT-LA&GD}), one gets  the following commutative diagram
  \begin{equation*}
\def\objectstyle{\displaystyle}
\xymatrix{
& 0 \ar[d] && 0 \ar[d] && 0 \ar[d] & \\
0 \ar[r] & M \ar[rr]^-{\left(
                                         \begin{array}{c}
                                           \mathbf{0}_{n,r}\,|\,\mathbf{I}_n  \\
                                         \end{array}
                                       \right)}\ar@{=}[d] &&
\mathcal{C}_T(Y)\cong\Pic(Y)\oplus M \ar[rr]^-{\left(
                                         \begin{array}{c}
                                           \mathbf{I}_r\,|\,\mathbf{0}_{r,n}  \\
                                         \end{array}
                                       \right)}\ar[d]^-{C_Y^T} && {\Pic(Y)} \ar[r]\ar[d]^-{B_Y^T} & 0 \\
0 \ar[r] & M \ar[rr]^-{div_Y}_-{\widehat{V}^T}\ar[d] && \mathcal{W}_T(Y)=\bigoplus_{j=1}^{n+r} \Z \cdot D_{j}\ar[d]
\ar[rr]^-{d_Y}_-Q && \Cl(Y)\ar[d] \ar[r] & 0 \\
 & 0\ar[rr] && \mathcal{T}_Y\ar[d]\ar[rr]^-{\cong} && \mathcal{T}_Y\ar[r]\ar[d]&0\\
 & &&0&&0& }
\end{equation*}
where $B_Y$ is the $r\times r$ matrix constructed in \cite[Thm.~2.9(3)]{RT-LA&GD} and
\begin{equation*}
    C_Y=\begin{pmatrix}B_Y & \mathbf{0}_{r,n}\\  \mathbf{0}_{n,r}& \mathbf{I}_{n}\end{pmatrix}\cdot U_Q= \begin{pmatrix}B_Y\cdot\,^rU_Q\\  \widehat{V}\end{pmatrix}\,,
\end{equation*}
\item once fixed the basis $\{D_j\}_{j=1}^{n+r}$ for $\mathcal{W}_T(X)\cong\Z^{n+r}$ and the basis $\{d_X(L_i)\}_{i=1}^r$ of the free part $F\cong\Z^r$ of $\Cl(X)$, constructed in part 1 of \cite[Thm.~3.2]{RT-QUOT}, one gets  the following commutative diagram
  \begin{equation*}
\def\objectstyle{\displaystyle}
\xymatrix{
& 0 \ar[d] && 0 \ar[d] && 0 \ar[d] & \\
0 \ar[r] & M \ar[rr]^-{\left(
                                         \begin{array}{c}
                                           \mathbf{0}_{n,r}\,|\,\mathbf{I}_n  \\
                                         \end{array}
                                       \right)}\ar@{=}[d] &&
\mathcal{C}_T(X)\cong\Pic(X)\oplus M \ar[rr]^-{\left(
                                         \begin{array}{c}
                                           \mathbf{I}_r\,|\,\mathbf{0}_{r,n}  \\
                                         \end{array}
                                       \right)}\ar[d]^-{C_X^T} && {\Pic(X)} \ar[r]\ar[d]^-{B_X^T\oplus \Theta_X^T} & 0 \\
0 \ar[r] & M \ar[rr]^-{div_X}_-{V^T}\ar[d] && \mathcal{W}_T(X)=\bigoplus_{j=1}^{n+r} \Z \cdot D_{j}\ar[d]
\ar[rr]^-{d_X=f_X\oplus\tau_X}_-{Q\oplus\Ga} && \Cl(X)\ar[d] \ar[r] & 0 \\
 & 0\ar[rr] && \mathcal{T}_X\ar[d]\ar[rr]^-{\cong} && \mathcal{T}_X\ar[r]\ar[d]&0\\
 & &&0&&0& }
\end{equation*}

\end{itemize}

Moreover:
\begin{itemize}
\item recall the following commutative diagram of short exact sequences
  \begin{equation}\label{div-diagram-covering}
    \begin{array}{c}
      \xymatrix{&&&0\ar[d]&\\
& 0 \ar[d] & 0 \ar[d] & \ker(\overline{\a})=\Tors(\Cl(X)) \ar[d] & \\
0 \ar[r] & M \ar[r]^-{div_X}_-{V^T}\ar[d]_-{\b^T} &
\mathcal{W}_T (X)=\Z^{|\Si(1)|} \ar[r]^-{d_X}\ar[d]^-{\a}_-{\mathbf{I}_{n+r}} & \Cl(X) \ar[r]\ar[d]^-{\overline{\a}} & 0 \\
0 \ar[r] & M \ar[r]^-{div_Y}_-{\widehat{V}^T}\ar[d]&\mathcal{W}_T(Y)=\Z^{|\widehat{\Si}(1)|}\ar[r]^-{d_Y}\ar[d] & \Cl (Y) \ar[r]\ar[d] & 0 \\
 & \coker(\b^T)\cong\Tors(\Cl(X))\ar[d] & 0 & 0 & \\
 &0&&&}
    \end{array}
\end{equation}
\end{itemize}
then, putting all together, one gets the following 3--dimensional commutative diagram
\begin{equation}\label{diagramma3D}
\begin{array}{c}
  \xymatrix{M\ar@{=}[dddd]\ar@{^{(}->}[rrr]^-{div_X}_-{\left(
                                         \begin{array}{c}
                                           \mathbf{0}_{n,r}\,|\,\mathbf{I}_n  \\
                                         \end{array}
                                       \right)}\ar@{^{(}->}[dr]^-{\b^T}&&&\Cart(X)\ar@{^{(}->}[dr]^-{\a_|}_>>>>>>{(C_X\cdot C_Y^{-1})^T}\ar@{->>}[rrr]^-{{d_X}_|}_-{\left(
                                         \begin{array}{c}
                                           \mathbf{I}_r\,|\,\mathbf{0}_{r,n}  \\
                                         \end{array}
                                       \right)}\ar@{^{(}->}[dddd]_>>>>>>>>>>>>>{C_X^T}&&&
  \Pic(X)\ar@{^{(}->}[dr]^-{\overline{\a}_|}_>>>>>>{(B_X\cdot B_Y^{-1})^T}\ar@{^{(}->}[dddd]\ar@{^{(}->}[dddd]_>>>>>>>>>>>>>>>{B_X^T\oplus\Theta_X^T}&&&&\\
             &M\ar@{^{(}->}[rrr]^-{div_Y}_-{\left(
                                         \begin{array}{c}
                                           \mathbf{0}_{n,r}\,|\,\mathbf{I}_n  \\
                                         \end{array}
                                       \right)}\ar@{=}[dddd]\ar@{->>}[dr]&&&\Cart(Y)\ar@{->>}[rrr]^-{{d_Y}_|}_-{\left(
                                         \begin{array}{c}
                                           \mathbf{I}_r\,|\,\mathbf{0}_{r,n}  \\
                                         \end{array}
                                       \right)}\ar@{^{(}->}[dddd]_-{C_Y^T}\ar@{->>}[dr]&&&\Pic(Y)\ar@{->>}[dr]\ar@{^{(}->}[dddd]_-{B_Y^T}&&\\
             &&\coker(\b^T)\ar@{=}[dddd]\ar@{^{(}->}[rrr]&&&\coker(\a_|)\ar@{->>}[rrr]&&&\coker(\overline{\a}_|)\\
             &&&&&\ker(\overline{\a})\ar@{^{(}->}[dddd]\ar@{^{(}->}[dr]&&&&\\
             M\ar@{^{(}->}[rrr]^-{div_X}_-{V^T}\ar@{^{(}->}[dr]^-{\b^T}&&&\Weil(X)\ar@{->>}[dddd]\ar@{->>}[rrr]^-{d_X=f_X\oplus\tau_X}_-{Q\oplus\Ga}\ar[dr]^-{\a}_{\mathbf{I}_{n+r}}&&&\Cl(X)\ar@{->>}[dr]^-{\overline{\a}}_-{\mathbf{I}_r\oplus\mathbf{0}_r}\ar@{->>}[dddd]&&&&\\
             &M\ar@{^{(}->}[rrr]^-{div_Y}_-{\widehat{V}^T}\ar@{->>}[dr]&&&\Weil(Y)\ar@{->>}[dddd]\ar@{->>}[rrr]^-{d_Y}_-{Q}&&&\Cl(Y)\ar@{->>}[dddd]&&\\
             &&\coker(\b^T)&&&&&&\\
             &&\mathcal{K}\ar@{^{(}->}[dr]\ar[rrr]_-{\cong}&&&\mathcal{K}\ar@{^{(}->}[dr]&&\\
             &&&\mathcal{T}_X\ar@{->>}[dr]\ar[rrr]_-{\cong}&&&\mathcal{T}_X\ar[dr]&&\\
             &&&&\mathcal{T}_Y\ar[rrr]_-{\cong}&&&\mathcal{T}_Y}
\end{array}
\end{equation}
The Snake Lemma implies
\begin{eqnarray*}
  \coker(\b^T)&\cong&\ker(\overline{\a})\cong\Tors(\Cl(X))\\
  \mathcal{K}&\cong&\coker(\a_|)\cong\Cart(Y)/\Cart(X)
\end{eqnarray*}
so giving the following short exact sequences on torsion subgroups
\begin{equation}\label{torsione}
  \xymatrix{&&0\ar[d]&\\
            0\ar[r]&\Tors(\Cl(X))\ar[r]&\Cart(Y)/\Cart(X)\ar[r]\ar[d]&\Pic(Y)/\Pic(X)\ar[r]&0\\
            &&\Cl(X)/\Pic(X)\ar[d]&\\
            &&\Cl(Y)/\Pic(Y)\ar[d]&\\
            &&0&}
\end{equation}
\end{itemize}}
\end{remark}

\setcounter{section}{5}
\setcounter{theorem}{0}
For what concerns the examples given in section 5, considerations related with parts v, vi and vii of Remark~\ref{rem:} have to be replaced as follows

\begin{example}\label{ex:K2}{\rm \hfill

\noindent v. A matrix $W\in\GL_4(\Z)$ such that $\HNF\left(({^3U})^T\right)=W\cdot ({^3U})^T$ is given by
\begin{equation*}
    W=\left(
          \begin{array}{cccc}
            1 & 0 & 0 & 0 \\
            1 & 0 & 1 & -2 \\
            0 & 1 & -3 & 2 \\
            0 & 0 & 1 & -1 \\
          \end{array}
        \right)
\end{equation*}
giving
\begin{eqnarray*}
      G &:=& {_1\widehat{V}}\cdot\ ({_{1}W})^T =\left(
                                                  \begin{array}{c}
                                                    1  \\
                                                  \end{array}
                                                \right)
       \\
    \end{eqnarray*}
Therefore
$$\Ga= {_{1}W} \mod 5= \left(
          \begin{array}{cccc}
            [0]_5 & [4]_5 & [2]_5 & [1]_5 \\
          \end{array}
        \right)\,.$$
Consequently display (16) in \cite{RT-QUOT}, giving the action of $\Hom(\Tors(\Cl(X)),\C^*)\cong \mu_5$ on $Y=\P^3$, should be replaced by the following (equivalent) one:
\begin{equation}\label{azione}
    \begin{array}{ccc}
       \mu_5\times\P^3 & \longrightarrow & \P^3 \\
       (\varepsilon,[x_1:\ldots :x_4]) & \,\mapsto & \left[x_1:\varepsilon^4 x_2:\varepsilon^2 x_3:\varepsilon x_4\right] \ .
     \end{array}
\end{equation}

\noindent vi. Applying procedure \cite[\S~1.2.3]{RT-LA&GD} as described in part 2 of Theorem \ref{thm:generazione}, one gets a $4\times 4$ matrix $C_X$ whose rows give a basis of $\Cart(X)$ inside $\Weil(X)\cong\Z^{|\Si(1)|}$. Namely
\begin{equation*}
    C_X=\left(
                                    \begin{array}{cccc}
                                      5 & 0 & 0 & 0 \\
                                      0 & 5 & 0 & 0 \\
                                      -3 & -3 & 1 & 0 \\
                                      -2 & -4 & 0 & 1 \\
                                    \end{array}
                                  \right)
\end{equation*}
meaning that $$\Cart(X)=\mathcal{L}\left( 5D_1,5D_2,-3D_1-3D_2+D_3,-2D_1-4D_2+D_4\right)\,.$$

\noindent On the other hand, by part (3) of \cite[Thm.~2.9]{RT-LA&GD}, a basis of
$\Cart(Y)\subseteq\Weil(Y)$ is given by the rows of
$$C_Y=\mathbf{I}_4\cdot U_Q=U_Q\in\GL_n(\Z)$$
giving $\Cart(Y)=\Weil(Y)$, as expected for $Y=\P^3$.

\noindent vii. A basis of $\Pic(X)$ inside $\Cl(X)$ is then obtained by applying part 6 of Theorem \ref{thm:generazione}. With the notation of Remark \ref{rem:} vii, a switching matrix $A$ such that $A\cdot C_X\cdot Q^T$ is in $\HNF$ is
$$A=\left(
                                    \begin{array}{cccc}
                                      1 & 0 & 0 & 0 \\
                                      -1 & 1 & 0 & 0 \\
                                      1 & 0 & 1 & 0 \\
                                      1 & 0 & 0 & 1 \\
                                    \end{array}
                                  \right)$$

\noindent so that

$$
B_X=\ ^1(A\cdot C_X\cdot Q^T) =\left(
      \begin{array}{c}
        5 \\
      \end{array}
    \right)
$$
$$
\Theta_X=\ ^1(A\cdot C_X\cdot \Gamma^T) =\left(
      \begin{array}{c}
        0 \\
      \end{array}
    \right)
$$
Then
$$\Pic(X)\cong \Z[5d_X(D_1)]\leq\Z[d_X(D_1)]\oplus\Z/5\Z[d_X(D_3-D_4)]\cong\Cl(X)\ \Rightarrow\ \Cl(X)/\Pic(X)\cong\Z/5\Z\oplus\Z/5\Z\,.
$$}
\end{example}

\begin{example}\label{ex:} {\rm \hfill

\noindent v. A matrix $U$ as defined in part 6 of Theorem \ref{thm:generazione} is given by
\begin{equation*}
    U=\left(
        \begin{array}{c}
          ^2U_Q \\
          \widehat{V}' \\
        \end{array}
      \right)
    =\left(
          \begin{array}{cccccc}
          2&-1&0&0&0&0\\
          -6&3&1&0&0&0\\
            521&-251&-168&-2&14&28\\
            388&-222&-112&7&45&3\\
            -184&105&53&-2&-23&-1\\
            191&-109&-55&2&24&1\\
          \end{array}
        \right)
\end{equation*}
A matrix $W\in\GL_6(\Z)$ such that $\HNF(({^4U})^T)=W\cdot(({^4U})^T)$ is given by
\begin{equation*}
    W=\left(
          \begin{array}{cccccc}
            -57&-115&3&-549&17&0\\
            4&8&1&3&7&0\\
            -125&-250&0&-1090&14&0\\
            -170&-340&0&-1482&19&0\\
            -188&-376&0&-1639&21&0\\
            -126&-252&0&-1092&13&1\\
          \end{array}
        \right)
\end{equation*}
then
\begin{equation*}
    G={_2\widehat{V}'}\cdot ({_2W})^T = \left(
                                          \begin{array}{cc}
                                            -2093&-1392\\
                                            2302&1531\\
                                          \end{array}
                                        \right)
\end{equation*}
A matrix $U_G\in\GL_2(\Z)$ such that $\HNF(G^T)=U_G\cdot G^T$ is given by
\begin{equation*}
    U_G=\left(
          \begin{array}{cc}
            1531&-2302\\
            1392&-2093\\
          \end{array}
        \right)
\end{equation*}
hence giving
\begin{eqnarray*}
  \Ga &=& {U_G}\ \cdot\ {_2W} \mod \boldsymbol\tau \\
  &=& \left(
                                                      \begin{array}{cccccc}
                                                        2224&4448&0&4475&2225&-2302\\ 2022&4044&0&4068&2023&-2093
                                                      \end{array}
                                                    \right)\mod \left(
                                                                  \begin{array}{c}
                                                                    3 \\
                                                                    15 \\
                                                                  \end{array}
                                                                \right)\\
   &=&
   \left(
          \begin{array}{cccccc}
          [1]_3&[2]_3   &[0]_3   &[2]_3   &[2]_3    &[2]_3\\
          {[12]_{15}} &[9]_{15}&[0]_{15}&[3]_{15}&[13]_{15}&[7]_{15}\\
          \end{array}
        \right)
\end{eqnarray*}
Consequently display (20) in \cite{RT-QUOT} should be replaced by the following (equivalent) one
\begin{eqnarray}\label{azione_g}
    &&g\left(((t_1,t_2),\varepsilon,\eta),(x_1,\ldots :x_6)\right):=\\
    \nonumber
    &&\left(t_1^2t_2\varepsilon\eta^{12}\ x_1,t_1^4t_2\varepsilon^2\eta^9\  x_2,t_1t_2^3 \ x_3, t_1^5t_2^2\varepsilon^2\eta^3\ x_4,t_1^4t_2^3\varepsilon^2\eta^{13}\ x_5,t_1^3t_2^7\varepsilon^2\eta^7\ x_6\right)
\end{eqnarray}

\noindent vi. Depending on the choice of the fan $\Si_i\in\SF(V)$, by applying procedure \cite[\S~1.2.3]{RT-LA&GD} as described in part 2 of Theorem \ref{thm:generazione}, one gets a $6\times 6$ matrix $C_{X,i}$ whose rows give a basis of $\Cart(X_i)$ inside $\Weil(X_i)\cong\Z^{|\Si_i(1)|}$. Namely
\begin{equation*}
    C_{X,1}=\left(\begin {array}{cccccc}
    265926375&0&0&0&0&0\\
    -148978500&825&0&0&0&0\\
    -58474020&-375&15&0&0&0\\
    37&-18&-7&1&0&0\\
    -58473933&-417&-3&0&3&0\\
    19&-8&-5&0&-1&1\end {array}
                                                                                                 \right)
\end{equation*}
\begin{equation*}
    C_{X,2}=\left( \begin {array}{cccccc}
    43543500&0&0&0&0&0\\
    -34716000&15&0&0&0&0\\
    -594165&0&30&0&0&0\\
    -34715963&-3&-7&1&0&0\\
    17655087&-12&-18&0&3&0\\
    19&-8&-5&0&-1&1
    \end {array} \right)
\end{equation*}
\begin{equation*}
    C_{X,3}=\left(\begin {array}{cccccc}
    43543500&0&0&0&0&0\\
    -37009500&825&0&0&0&0\\
    -6534165&-750&30&0&0&0\\
    37&-18&-7&1&0&0\\
    87&-42&-18&0&3&0\\
    19&-8&-5&0&-1&1\end {array}
                                                                                                 \right)
\end{equation*}

\noindent vii. A basis of $\Pic(X_i)$ inside $\Cl(X_i)$ is then  obtained by applying part 6 of Theorem \ref{thm:generazione}. For $i=1,2,3$, matrices $A_i$ switching  $C_{X_i}\cdot Q^T$ in Hermite normal form are respectively
\begin{equation*}
A_1=\left(
      \begin{array}{cccccc}
       -351039&-449987&-449987&0&0&0\\
       -502913&-644670&-644670&0&0&0\\
       1&1&2&0&0&0\\
       0&0&0&1&0&0\\
       1&1&1&0&1&0\\
       0&0&0&0&0&1
      \end{array}
    \right)\end{equation*}

\begin{equation*}
A_2=\left(
      \begin{array}{cccccc}
       -93838&-117699&0&0&0&0\\
       -1157199&-1451450&0&0&0&0\\
       4&5&1&0&0&0\\
       0&-1&0&1&0&0\\
       -2&-2&0&0&1&0\\
       0&0&0&0&0&1
      \end{array}
    \right)\end{equation*}

    \begin{equation*}
    A_3=\left(
          \begin{array}{cccccc}
          -10317&-12139&0&0&0&0\\
          -22429&-26390&0&0&0&0\\
          1&1&1&0&0&0\\0&0&0&1&0&0\\
          0&0&0&0&1&0\\
          0&0&0&0&0&1
          \end{array}
        \right)\end{equation*}

giving
\begin{eqnarray*}
  B_{X_1} &=& \ ^2(A_1\cdot C_{X_1}\cdot Q^T) =\left(
      \begin{array}{cc}
       825&185620050\\
      0&265926375
      \end{array}
    \right) \\
  B_{X_2} &=&  \ ^2(A_2\cdot C_{X_2}\cdot Q^T) =\left(
      \begin{array}{cc}
       60&1765515\\
       0&21771750
      \end{array}
    \right) \\
  B_{X_3} &=&  \ ^2(A_3\cdot C_{X_3}\cdot Q^T) =\left(
      \begin{array}{cc}
       3300&10016325\\
       0&21771750
      \end{array}
    \right)\\
  \Theta_{X_i}&=&   \ ^2(A_i\cdot C_{X_i}\cdot \Ga^T) =\left(
                                                         \begin{array}{cc}
                                                           \,[0]_3 & [0]_{15} \\
                                                           \,[0]_3 & [0]_{15} \\
                                                         \end{array}
                                                       \right)\ ,\quad\text{for $i=1,2,3$\,.}
\end{eqnarray*}
}
\end{example}

\end{document}